\documentclass[letterpaper, oneside]{amsart}

\usepackage{layout}	

\usepackage[
]{geometry}

\usepackage[parfill]{parskip}   
\usepackage{amssymb}
\usepackage{amsthm}
\usepackage{amsmath}
\usepackage{cite}
\usepackage[all]{xy}
\usepackage{enumerate}
\usepackage{verbatim}
\usepackage{mathtools}
\usepackage{hyperref}
\usepackage{colonequals}
\usepackage{mathrsfs}
\usepackage{color}
\usepackage{adjustbox}

\usepackage{tikz}
\usetikzlibrary{decorations.markings}
\usetikzlibrary{arrows.meta}
\tikzset{
    triple/.style args={[#1] in [#2] in [#3]}{
        #1,preaction={preaction={draw,#3},draw,#2}
    }
}      

\usepackage{longtable}
\usepackage{multicol}
\usepackage{multirow}

\usepackage{thmtools}

\theoremstyle{plain}
\newtheorem{thm}{Theorem}[section]
\newtheorem{lem}[thm]{Lemma}

\theoremstyle{definition}

\declaretheorem[name=Remark, style=definition, qed={\lower-0.3ex\hbox{$\blacksquare$}}, unnumbered]{rmk}
\newtheorem*{ack}{Acknowledgement}

\newtheorem{conj}[thm]{Conjecture}

\numberwithin{equation}{section}

\newcommand{\A}{\mathbb{A}}

\newcommand{\Z}{\mathbb{Z}}
\newcommand{\N}{\mathbb{N}}

\let\P\relax
\newcommand{\P}{\mathbb{P}}
\newcommand{\C}{\mathbb{C}}

\newcommand{\X}{\mathbb{X}}

\let\k\relax
\newcommand{\k}{\mathbf{k}}

\let\O\relax
\newcommand{\O}{\mathcal{O}}
\newcommand{\Y}{\mathcal{Y}}
\newcommand{\B}{\mathcal{B}}

\let\L\relax
\newcommand{\L}{\mathcal{L}}

\let\S\relax
\newcommand{\S}{\mathcal{S}}

\newcommand{\g}{\mathfrak{g}}
\newcommand{\h}{\mathfrak{h}}
\let\b\relax
\newcommand{\b}{\mathfrak{b}}

\newcommand{\n}{\mathfrak{n}}

\newcommand{\br}[1]{\left\langle{#1}\right\rangle}

\newcommand{\scvar}[1]{\overline{\mathcal{S}(#1)}}

\newcommand{\nn}{N}

\newcommand{\iA}{5}
\newcommand{\iB}{1}
\newcommand{\iC}{4}
\newcommand{\iD}{2}
\newcommand{\iE}{3}
\newcommand{\iF}{6}

\DeclareMathOperator{\im}{\textup{im}}

\DeclareMathOperator{\ad}{\textup{ad}}
\DeclareMathOperator{\Ad}{\textup{Ad}}

\DeclareMathOperator{\ch}{\textup{char}}
\DeclareMathOperator{\Irr}{\textup{Irr}}

\DeclareMathOperator{\Lie}{\textup{Lie}}
\DeclareMathOperator{\codim}{\textup{codim}}
\DeclareMathOperator{\Pic}{\textup{Pic}}

\title{Springer fibers for the minimal and the minimal special nilpotent orbits}
\author{Dongkwan Kim}
\address{Department of Mathematics\\
  Massachusetts Institute of Technology\\
  Cambridge, MA 02139-4307\\
  U.S.A.}
\email{sylvaner@math.mit.edu}
\date{\today}							

\begin{document}
\begin{abstract} We describe Springer fibers corresponding to the minimal and minimal special nilpotent orbits of simple Lie algebras. As a result, we give an answer to the conjecture of Humphreys regarding some graphs attached to Springer fibers.
\end{abstract}

\setcounter{tocdepth}{1}
\maketitle

\renewcommand\contentsname{}
\tableofcontents

\section{Introduction}
Let $G$ be a simple algebraic group over an algebraically closed field $\k$ and let $\g$ be its Lie algebra. We assume that $\ch \k$ is good. Let $B$ be a Borel subgroup of $G$ and $\b$ be its corresponding Lie algebra, which is by definition a Borel subalgebra of $\g$. For a nilpotent element $\nn \in \g$, we consider
$$\B_\nn \colonequals \{gB \in G/B \mid \Ad g^{-1}(\nn) \in \b\},$$
which is called the Springer fiber of $\nn$. Since the celebrated paper of Springer \cite{springer:trig}, it has become one of the central objects in geometric representation theory.

This paper mainly considers the case when $\nn$ is contained in the minimal nilpotent orbit or the minimal special nilpotent orbit of $\g$. Here we say a nilpotent orbit of $\g$ is minimal if it is minimal among non-trivial nilpotent orbits with respect to the order defined by $\O \leq \O' \Leftrightarrow \O \subset \overline{\O'}$. Similarly we say a nilpotent orbit of $\g$ is minimal special if it is minimal among non-trivial special nilpotent orbits.

For $\nn \in \g$ contained in the minimal nilpotent orbit of $\g$, the corresponding Springer fiber $\B_\nn$ is studied by Dolgachev and Goldstein \cite{dg:minimal}. In their paper a graph $\Gamma_\nn$ attached to $\B_\nn$ is introduced (originally defined by Kazhdan and Lusztig \cite[Section 6.3]{kl:hecke}): its vertices are the irreducible components of $\B_\nn$ and two vertices are connected by an edge if and only if their intersection is of codimension 1 in each irreducible component. In this paper we extend their results to include the case when $\nn$ is contained in the minimal special nilpotent orbit. (Note that the minimal and the minimal special nilpotent orbit coincide if and only if $G$ is of simply-laced type.)

This paper is motivated by the conjecture of Humphreys \cite{humphreys:conj}. (See also \cite{humphreys:mof} and \cite{humphreys:minimal}.) In \cite{dg:minimal} Dolgachev and Goldstein observed a duality between the graphs attached to Springer fibers corresponding to the subregular and the minimal nilpotent orbit when $G$ is of simply-laced type, but this duality breaks down when $G$ is not simply-laced. Instead, Humphreys conjectured that one needs to consider the minimal special nilpotent orbit instead of the minimal one, and also that the Langlands dual should come into the picture. (Indeed, the idea to consider minimal special nilpotent orbits already appeared in the last paragraph of \cite[p.34]{dg:minimal} following the remark of Spaltenstein.)

The main goal of this paper is to describe $\Gamma_\nn$ for all simple $\g$ and all minimal/minimal special nilpotent orbits. As a result, we verify the conjecture of Humphreys as follows.

\begin{thm}[Main theorem]\label{thm:main} Let $G$ be a simple algebraic group over an algebraically closed field $\k$ and let $\nn \in \g \colonequals \Lie G$ be contained in the minimal special nilpotent orbit of $\g$. (We assume $\ch \k$ is good.) Then $\Gamma_\nn$ is isomorphic to the graph corresponding to a subregular nilpotent element in the Lie algebra of the Langlands dual of $G$.
\end{thm}

In fact, this theorem has a generalization to some special cases in type $A$. In \cite{fresse:unified} Fresse observed the following; let $G$ be of type $A$ and $N_\lambda \in \g$ be a nilpotent element which corresponds to a partition $\lambda$. If $\lambda$ is a hook, two-row, or two-column partition, then $\Gamma_{N_\lambda}$ is isomorphic to $\Gamma_{N_{\lambda^t}}$, where ${\lambda^t}$ is the transpose of $\lambda$. Note that $N_\lambda$ and $N_{\lambda^t}$ are related by the Lusztig-Spaltenstein duality on the set of special nilpotent orbits in $\g$. In general it is known that for such a pair of special nilpotent orbits dual to each other, their Springer representations have the same dimension. Thus it is natural to ask the following question.

\begin{conj} \label{conj} For $N \in \g$ contained in a special nilpotent orbit, the quotient of $\Gamma_N$ by the action of the component group of the stabilizer of $N$ in $G$ is isomorphic to that corresponding to the Lusztig-Spaltenstein dual of the orbit.
\end{conj}

Also, there is an order preserving bijection between the sets of special nilpotent orbits in $\Lie(SO_{2n+1})(\k)$ and $\Lie(Sp_{2n})(\k)$. Since such two orbits related by this bijection have the same irreducible Springer representation, the following question is also natural to ask.
\begin{conj} \label{conj} Let $N \in \Lie(SO_{2n+1})(\k)$, $N' \in \Lie(Sp_{2n})(\k)$ be contained in special nilpotent orbits related by the order-preserving bijection described above. Then the quotient of $\Gamma_N$ by the action of the component group of the stabilizer of $N$ in $G$ is isomorphic to that corresponding to $N'$.
\end{conj}

In this paper, for each $N \in \g$ in the minimal and the minimal special nilpotent orbit we also analyze the action on $\Gamma_N$ of the component group of the stabilizer of $N$ in $G$, and thus one can easily verify that Conjecture \ref{conj} holds in such cases. We also give an explicit description of each irreducible component for each such Springer fiber. As a result, in most cases we check whether all components of each Springer fiber are smooth or not, except when $G$ is of type $F_4$ and $\nn$ is contained in the minimal special nilpotent orbit.

We remark that a lot of calculations for exceptional types are done using SageMath \cite{sagemath}, and most of the calculations to obtain Kazhdan-Lusztig polynomials are done using Coxeter 64 \cite{coxeter}.

\begin{ack} The author is grateful to Roman Bezrukavnikov who brought this topic to his attention and made valuable suggestions to him. He also thanks George Lusztig and David Vogan for helpful remarks. Special thanks are due to Jim Humphreys for his careful reading of drafts of this paper, pointing out several errors, and giving thoughtful comments.
\end{ack}

\section{Preliminaries}\label{sec:prelim}
Let $\k$ be an algebraically closed field. For a variety $X$ over $\k$, we write $\Irr(X)$ to be the set of irreducible components of $X$. If $Y$ is a subvariety of $X$, then $\overline{Y} \subset X$ denotes the closure of $Y$ in $X$. For varieties $X$, $Y$, and $Z$ over $\k$, we write $Z:X\dashleftarrow Y$ if $Z$ is (algebraically) a locally trivial fiber bundle over $X$ with fiber $Y$. It does not describe the fiber bundle structure explicitly, but note that $Z$ is smooth if $X$ and $Y$ are smooth. Also for varieties $X_i$ and $Z$, $Z:X_1 \dashleftarrow X_2 \dashleftarrow \cdots \dashleftarrow X_n$ means $Z:((X_1 \dashleftarrow X_2) \dashleftarrow \cdots) \dashleftarrow X_n$.

Let $Gr(k, n)$ be the Grassmannian of $k$-dimensional subspaces in $\k^n$. Similarly, when $\ch \k \neq 2$ we denote by $OGr(k,n)$ the Grassmannian of $k$-dimensional isotropic subspaces in $\k^n$ equipped with a non-degenerate symmetric bilinear form. Likewise, when $\ch \k \neq 2$ we define $SGr(k,n)$ to be the Grassmannian of $k$-dimensional isotropic subspaces in $\k^n$ equipped with a non-degenerate symplectic bilinear form. The last one is only well-defined when $n$ is even, but when $n$ is odd we still write $SGr(k,n)$ to denote an ``odd'' symplectic Grassmannian in the sense of \cite{mihai}. Then $Gr(k,n), OGr(k,n), SGr(k,n)$ are smooth of dimension $k(n-k), \frac{1}{2}k(2n-3k-1), \frac{1}{2}k(2n-3k+1)$, respectively. (In the odd symplectic case, it follows from \cite[Proposition 4.1]{mihai}.)

Let $G$ be a simple algebraic group over $\k$. We shall assume $\ch\k$ is good for $G$. We fix a Borel subgroup $B\subset G$ and its maximal torus $T \subset B$. We denote by $\Phi$ the set of roots of $G$ with respect to $T$, and $\Phi^+ \subset \Phi$ be the set of positive roots with respect to the choice of $T \subset B$. Then the set of simple roots $\Pi \subset \Phi^+$ is well-defined. For any $\alpha \in \Phi$, let $U_\alpha \subset G$ be the one-parameter subgroup corresponding to $\alpha$, which is isomorphic to $\A^1_\k$ as a group. We have a decomposition
$$B=T \cdot \prod_{\alpha \in \Phi^+} U_\alpha$$
and $\prod_{\alpha \in \Phi^+} U_\alpha \subset B$ is its unipotent radical, denoted $U$.

Let $\g \colonequals \Lie G$ be the Lie algebra of $G$, $\b \colonequals \Lie B \subset \g$ be the Borel subalgebra of $\g$ corresponding to $B$, and $\h \colonequals \Lie T \subset \b$ be the Cartan subalgebra corresponding to $T$. Also, for any $\alpha \in \Phi$, we let $X_\alpha \colonequals \Lie U_\alpha \subset \b$ be the Lie subalgebra corresponding to $U_\alpha$. Then we have a decomposition
$$\b = \h \oplus \bigoplus_{\alpha \in \Phi^+} X_\alpha $$
and $\bigoplus_{\alpha \in \Phi^+} X_\alpha=\Lie U \subset \b$ is the nilpotent radical of $\b$, denoted $\n$.

Let $W$ be the Weyl group of $G$, which is the quotient of the normalizer of $T\subset G$ by $T$. Then there exists a set of simple reflections $S \subset W$ which corresponds to $\Pi \subset \Phi^+$, and $(W,S)$ becomes a Coxeter group. Also there exists a well-defined length function $l: W \rightarrow \N$ where $l(w)$ is the length of some/any reduced expression of $w \in W$ with respect to $S$.

Let $\B \colonequals G/B$ be the flag variety of $G$. We also use the notation $\B(G)$ if we need to specify the group $G$. For any $s \in S$, we define a line of type $s$ to be 
$$\overline{gBsB/B} = gBsB/B \sqcup gB/B \subset G/B$$ for some $g \in G$. Thus if we let $P_s\subset G$ be a parabolic subgroup of $G$ generated by $\br{B, s}$, then a line of type $s$ is $gP_s/B \subset G/B$ for some $g\in G$.

For each $w\in W$, we define $\S(w) \colonequals BwB/B \subset \B$ to be the (open) Schubert cell corresponding to $w$. Then we have a Bruhat decomposition
$$\B = \bigsqcup_{w\in W} BwB/B$$
which gives a stratification of $\B$ into affine spaces. Also the (closed) Schubert variety of $w$ is $\scvar{w}$.

For any element $\nn \in \g$, we define $\B_\nn$ to be the Springer fiber of $\nn$, i.e.
$$\B_\nn \colonequals \{gB/B \in G/B \mid \Ad{g}^{-1}(\nn) \in \b\}.$$
In most cases $\nn \in \g$ is a nilpotent element in this paper. Attached to $\B_\nn$ there exists a graph $\Gamma_\nn$ which is defined as follows. The set of vertices of $\Gamma_\nn$, denoted $V(\Gamma_\nn)$, is simply $\Irr(\B_\nn)$. The set of edges of $\Gamma_\nn$, denoted $E(\Gamma_\nn)$, contains $(X,Y)$ for some $X, Y \in \Irr(\B_\nn)$ if and only if $X\cap Y$ is of codimension 1 in $X$ and $Y$, i.e. $\dim X\cap Y = \dim X-1=\dim Y-1$. (Note that $\dim X=\dim Y=\dim \B_\nn$ by \cite{spaltenstein:equidim}.) Also for any $X \in V(\Gamma_\nn)=\Irr(\B_\nn)$ we define
$$I_X \colonequals \{s \in S \mid X \textup{ is a union of lines of type } s\}.$$
Let $C(\nn)$ be the component group of the stabilizer of $\nn$ in $G$. Then $C(\nn)$ naturally acts on $\Irr \B_\nn$ and $\Gamma_\nn$. Also for any $\sigma \in C(\nn)$ and $X\in \Irr \B_\nn$, $I_{X} = I_{\sigma(X)}$.

We label the simple roots of $G$ by $\alpha_1, \cdots, \alpha_n$ where $n$ is the rank of $G$, such that it agrees with the following Dynkin diagrams. Also we label $S=\{s_1, \cdots, s_n\}$ where each $s_i$ corresponds to $\alpha_i$. For short, we write $s(i_1i_2\cdots i_r) \colonequals s_{i_1}s_{i_2}\cdots s_{i_r}$.

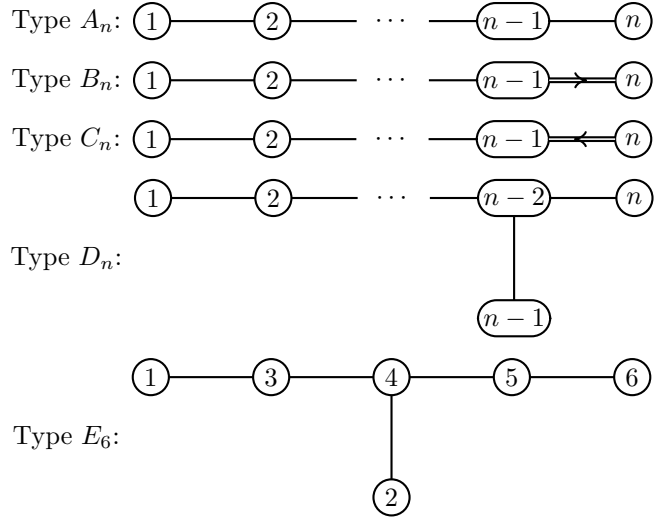
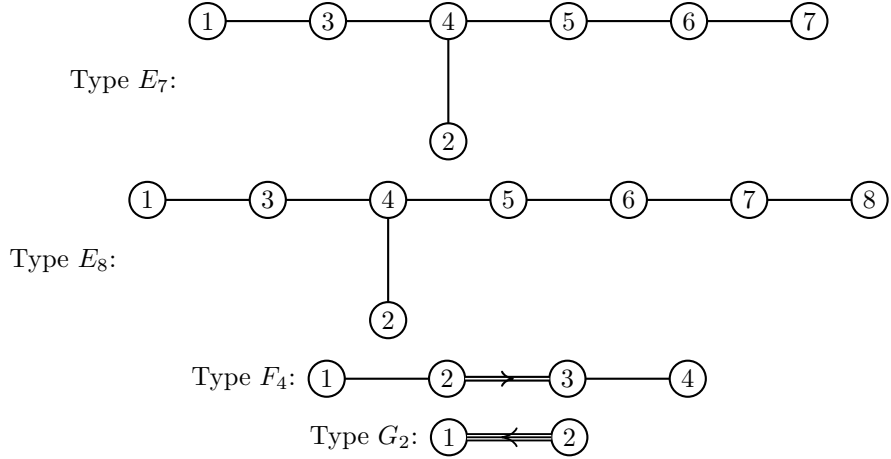
\begin{center}
Type $A_n$:
\begin{adjustbox}{valign=c}
\begin{tikzpicture}[scale=0.8]
	\node at (0 cm,0) {$1$};
	\node at (2 cm,0) {$2$};
	\node at (4 cm,0) {$\cdots$};
	\node at (6 cm,0) {$n-1$};
	\node at (8 cm,0) {$n$};
	
    \draw[thick, rounded corners=7] (5.4, -0.3) rectangle (6.6, 0.3) {};

    \draw[thick] (0 cm,0) circle (3 mm);
    \draw[thick] (2 cm,0) circle (3 mm);
    \draw[thick] (8 cm,0) circle (3 mm);
    
    \draw[thick] (0.3, 0) -- (1.7,0);
    \draw[thick] (2.3, 0) -- (3.4,0);
    \draw[thick] (4.6, 0) -- (5.4,0);
    \draw[thick] (6.6, 0) -- (7.7,0);

\end{tikzpicture}
\end{adjustbox}

Type $B_n$:
\begin{adjustbox}{valign=c}
\begin{tikzpicture}[scale=0.8]
	\node at (0 cm,0) {$1$};
	\node at (2 cm,0) {$2$};
	\node at (4 cm,0) {$\cdots$};
	\node at (6 cm,0) {$n-1$};
	\node at (8 cm,0) {$n$};
    \draw[thick, rounded corners=7] (5.4, -0.3) rectangle (6.6, 0.3) {};
    
    \draw[thick] (0 cm,0) circle (3 mm);
    \draw[thick] (2 cm,0) circle (3 mm);
    \draw[thick] (8 cm,0) circle (3 mm);
    
    \draw[thick] (0.3, 0) -- (1.7,0);
    \draw[thick] (2.3, 0) -- (3.4,0);
    \draw[thick] (4.6, 0) -- (5.4,0);
    
    \begin{scope}[decoration={markings,mark=at position 0.6 with {\arrow{>[scale=2.5,
          length=2.5,
          width=3]}}}] 
    	\draw[postaction={decorate},thick,double,] (6.6, 0) -- (7.7,0);
    \end{scope}

\end{tikzpicture}
\end{adjustbox}

Type $C_n$:
\begin{adjustbox}{valign=c}
\begin{tikzpicture}[scale=0.8]
	\node at (0 cm,0) {$1$};
	\node at (2 cm,0) {$2$};
	\node at (4 cm,0) {$\cdots$};
	\node at (6 cm,0) {$n-1$};
	\node at (8 cm,0) {$n$};
    \draw[thick, rounded corners=7] (5.4, -0.3) rectangle (6.6, 0.3) {};
    
    \draw[thick] (0 cm,0) circle (3 mm);
    \draw[thick] (2 cm,0) circle (3 mm);
    \draw[thick] (8 cm,0) circle (3 mm);
    
    \draw[thick] (0.3, 0) -- (1.7,0);
    \draw[thick] (2.3, 0) -- (3.4,0);
    \draw[thick] (4.6, 0) -- (5.4,0);
    
    \begin{scope}[decoration={markings,mark=at position 0.6 with {\arrow{<[scale=2.5,
          length=2.5,
          width=3]}}}] 
    	\draw[postaction={decorate},thick,double,] (6.6, 0) -- (7.7,0);
    \end{scope}

\end{tikzpicture}
\end{adjustbox}

Type $D_n$:
\begin{adjustbox}{valign=c}
\begin{tikzpicture}[scale=0.8]
	\node at (0 cm,0) {$1$};
	\node at (2 cm,0) {$2$};
	\node at (4 cm,0) {$\cdots$};
	\node at (6 cm,0) {$n-2$};
	\node at (6 cm,-2) {$n-1$};
	\node at (8 cm,0) {$n$};
    \draw[thick, rounded corners=7] (5.4, -0.3) rectangle (6.6, 0.3) {};
    \draw[thick, rounded corners=7] (5.4, -2.3) rectangle (6.6, -1.7) {};
    
    \draw[thick] (0 cm,0) circle (3 mm);
    \draw[thick] (2 cm,0) circle (3 mm);
    \draw[thick] (8 cm,0) circle (3 mm);
    
    \draw[thick] (0.3, 0) -- (1.7,0);
    \draw[thick] (2.3, 0) -- (3.4,0);
    \draw[thick] (4.6, 0) -- (5.4,0);
    \draw[thick] (6.6, 0) -- (7.7,0);
    \draw[thick] (6, -0.3) -- (6,-1.7);

\end{tikzpicture}
\end{adjustbox}

Type $E_6$:
\begin{adjustbox}{valign=c}\begin{tikzpicture}[scale=0.8]
	\node at (0 cm,0) {$1$};
	\node at (2 cm,0) {$3$};
	\node at (4 cm,0) {$4$};
	\node at (6 cm,0) {$5$};
	\node at (8 cm,0) {$6$};
	\node at (4 cm,-2cm) {$2$};

    \draw[thick] (0 cm,0) circle (3 mm);
    \draw[thick] (2 cm,0) circle (3 mm);
    \draw[thick] (4 cm,0) circle (3 mm);
    \draw[thick] (6 cm,0) circle (3 mm);
    \draw[thick] (8 cm,0) circle (3 mm);
    \draw[thick] (4 cm,-2cm) circle (3 mm);    
    
    \draw[thick] (0.3, 0) -- (1.7,0);
    \draw[thick] (2.3, 0) -- (3.7,0);
    \draw[thick] (4.3, 0) -- (5.7,0);
    \draw[thick] (6.3, 0) -- (7.7,0);
    \draw[thick] (4, -0.3) -- (4,-1.7);
\end{tikzpicture}
\end{adjustbox}

Type $E_7$:
\begin{adjustbox}{valign=c}\begin{tikzpicture}[scale=0.8]
	\node at (0 cm,0) {$1$};
	\node at (2 cm,0) {$3$};
	\node at (4 cm,0) {$4$};
	\node at (6 cm,0) {$5$};
	\node at (8 cm,0) {$6$};
	\node at (10 cm,0) {$7$};
	\node at (4 cm,-2cm) {$2$};

    \draw[thick] (0 cm,0) circle (3 mm);
    \draw[thick] (2 cm,0) circle (3 mm);
    \draw[thick] (4 cm,0) circle (3 mm);
    \draw[thick] (6 cm,0) circle (3 mm);
    \draw[thick] (8 cm,0) circle (3 mm);
    \draw[thick] (10 cm,0) circle (3 mm);
    \draw[thick] (4 cm,-2cm) circle (3 mm);    
    
    \draw[thick] (0.3, 0) -- (1.7,0);
    \draw[thick] (2.3, 0) -- (3.7,0);
    \draw[thick] (4.3, 0) -- (5.7,0);
    \draw[thick] (6.3, 0) -- (7.7,0);
    \draw[thick] (8.3, 0) -- (9.7,0);
    \draw[thick] (4, -0.3) -- (4,-1.7);
\end{tikzpicture}
\end{adjustbox}

Type $E_8$:
\begin{adjustbox}{valign=c}
\begin{tikzpicture}[scale=0.8]
	\node at (0 cm,0) {$1$};
	\node at (2 cm,0) {$3$};
	\node at (4 cm,0) {$4$};
	\node at (6 cm,0) {$5$};
	\node at (8 cm,0) {$6$};
	\node at (10 cm,0) {$7$};
	\node at (12 cm,0) {$8$};
	\node at (4 cm,-2cm) {$2$};

    \draw[thick] (0 cm,0) circle (3 mm);
    \draw[thick] (2 cm,0) circle (3 mm);
    \draw[thick] (4 cm,0) circle (3 mm);
    \draw[thick] (6 cm,0) circle (3 mm);
    \draw[thick] (8 cm,0) circle (3 mm);
    \draw[thick] (10 cm,0) circle (3 mm);
    \draw[thick] (12 cm,0) circle (3 mm);
    \draw[thick] (4 cm,-2cm) circle (3 mm);    
    
    \draw[thick] (0.3, 0) -- (1.7,0);
    \draw[thick] (2.3, 0) -- (3.7,0);
    \draw[thick] (4.3, 0) -- (5.7,0);
    \draw[thick] (6.3, 0) -- (7.7,0);
    \draw[thick] (8.3, 0) -- (9.7,0);
    \draw[thick] (10.3, 0) -- (11.7,0);
    \draw[thick] (4, -0.3) -- (4,-1.7);
\end{tikzpicture}
\end{adjustbox}

Type $F_4$:
\begin{adjustbox}{valign=c}
\begin{tikzpicture}[scale=0.8]
	\node at (0 cm,0) {$1$};
	\node at (2 cm,0) {$2$};
	\node at (4 cm,0) {$3$};
	\node at (6 cm,0) {$4$};

    \draw[thick] (0 cm,0) circle (3 mm);
    \draw[thick] (2 cm,0) circle (3 mm);
    \draw[thick] (4 cm,0) circle (3 mm);
    \draw[thick] (6 cm,0) circle (3 mm);
    
    \draw[thick] (0.3, 0) -- (1.7,0);
    \begin{scope}[decoration={markings,mark=at position 0.6 with {\arrow{>[scale=2.5,
          length=2.5,
          width=3]}}}] 
    	\draw[postaction={decorate},thick,double,] (2.3, 0) -- (3.7,0);
    \end{scope}
    \draw[thick] (4.3, 0) -- (5.7,0);
\end{tikzpicture}
\end{adjustbox}

Type $G_2$:
\begin{adjustbox}{valign=c}
\begin{tikzpicture}[scale=0.8]
	\node at (0 cm,0) {$1$};
	\node at (2 cm,0) {$2$};

    \draw[thick] (0 cm,0) circle (3 mm);
    \draw[thick] (2 cm,0) circle (3 mm);
    
    \begin{scope}[decoration={markings,mark=at position 0.6 with {\arrow{<[scale=2.5,
          length=2.5,
          width=3]}}}] 
    	\draw[postaction={decorate},thick,triple={[line width=0.25mm,black] in
      [line width=0.6mm,white] in
      [line width=1.1mm,black]}] (0.3, 0) -- (1.7,0);
    \end{scope}
\end{tikzpicture}
\end{adjustbox}
\end{center}

We define the minimal nilpotent orbit of $\g$ to be the nilpotent orbit of $\g$ which is minimal among non-trivial nilpotent orbits where the order is given by $\O \leq \O' \Leftrightarrow \O \subset \overline{\O'}$. Also we define the minimal special nilpotent orbit of $\g$ to be the nilpotent orbit which is minimal among non-trivial special nilpotent orbits with the same order. (Here the term ``special" is in the sense of Lusztig.) These orbits are always well-defined and coincide if and only if $G$ is of simply-laced type, i.e. type $A$, $D$, $E$. Here we summarize some properties of these orbits and corresponding Springer fibers. These can be read from subsequent sections of this paper, \cite{dg:minimal}, \cite{carter:book}, \cite{cm:nilpotent}, etc.
\begin{center}
\begin{tabular}{|c|c|c|c|c|c|}
\hline
Type&$\nn \in \g$&M/MS&$\dim \B_\nn$& $|\Irr \B_\nn|$&$C(\nn)$
\\\hline$A_n$&$(1^{n-1}2^1)$&M=MS&$\frac{1}{2}(n^2-n)$&$n$&1
\\\hline\multirow{2}{*}{$B_n$}&$(1^{2n-3}2^2)$&M&$n^2-2n+2$&$n-1$&1
\\\cline{2-6}&$(1^{2n-2}3^1)$&MS&$n^2-2n+1$&$n+1$&$\Z/2$
\\\hline\multirow{2}{*}{$C_n$}&$(1^{2n-2}2^1)$&M&$n^2-n$&1&1
\\\cline{2-6}&$(1^{2n-4}2^2)$&MS&$n^2-2n+1$&$2n-1$&$\Z/2$
\\\hline$D_n$&$(1^{2n-4}2^2)$&M=MS&$n^2-3n+3$&$n$&1
\\\hline$E_6$&$A_1$&M=MS&25&$6$&1
\\\hline$E_7$&$A_1$&M=MS&46&$7$&1
\\\hline$E_8$&$A_1$&M=MS&91&$8$&1
\\\hline\multirow{2}{*}{$F_4$}&$A_1$&M&16&2&1
\\\cline{2-6}&$\widetilde{A_1}$&MS&13&6&$\Z/2$
\\\hline\multirow{2}{*}{$G_2$}&$A_1$&M&$3$&1&1
\\\cline{2-6}&$G_2(a_1)$&MS&$1$&4&$S_3$
\\\hline
\end{tabular}
\end{center}
Here each column represents:
\begin{enumerate}[$\bullet$]
\item Type: the type of $G$
\item $\nn \in \g$: if $G$ is of classical type, then it means the Jordan type of $\nn$. Otherwise, it is the type of the distinguished parabolic subgroup attached to $\nn$ with respect to the Bala-Carter classification.
\item M/MS: ``M'' if $\nn$ is in the minimal nilpotent orbit, ``MS'' if $\nn$ is in the minimal special nilpotent orbit, and ``M=MS'' if $\nn$ is in the minimal nilpotent orbit which is also special.
\item $\dim \B_\nn$: dimension of the Springer fiber of $\nn$
\item $|\Irr \B_\nn|$: the number of irreducible components of $\B_\nn$
\item $C(\nn)$: the component group of the stabilizer of $\nn$ in $G$ when $G$ is simple of adjoint type\footnote{It is known that the action of the component group of the stabilizer of $\nn$ on $\Gamma_\nn$ factors through that in the case when $G$ is of adjoint type, thus it suffices to consider $C(\nn)$ for adjoint $G$.}
\end{enumerate}

In this paper, we are mainly interested in the following properties.
\begin{enumerate}[$(1)$]
\item explicit descriptions of each irreducible component of $\B_\nn$
\item whether all irreducible components of $\B_\nn$ are smooth or not
\item $I_X$ for each $X \in \Irr \B_\nn$
\item description of $\Gamma_\nn$ and the action of $C(\nn)$ on $\Gamma_\nn$
\end{enumerate}

\section{Classical types}
In this section we assume that $G$ is of classical type and analyze $\B_\nn$ when $\nn \in \g$ is contained in either the minimal or minimal special nilpotent orbit. Note that for minimal nilpotent orbits it is already studied in \cite{dg:minimal}. Thus here we skip proofs which can be read in \cite{dg:minimal}, or if it follows from direct calculation, unless nontrivial observation is needed.

\subsection{Type $A$} \label{sec:typeA}
Assume $G=SL_{n+1}(\k)$ which is of type $A_{n}$. Then there is a one-to-one correspondence between nilpotent orbits of $\g$ and partitions of $n+1$ given by taking the sizes of Jordan blocks of a nilpotent element. Under this correspondence the minimal nilpotent orbit is described by the partition $(1^{n-1}2^1)$ and it is also special. 

We regard $G$ as the set of linear automorphisms of determinant 1 acting on $\k^{n+1}$.  
Then the flag variety $\B$ is identified with the set of full flags in $\k^{n+1}$, i.e.
$$\B = \{ [F_i]=[F_0 \subset F_1 \subset \cdots \subset F_{n}\subset F_{n+1}=\k^{n+1}] \mid  \dim F_i = i\}.$$
Let $s_i\in S$ be chosen so that for any $[F_j] \in \B$ and $s_i \in S$, the line of type $s_i$ passing through $[F_j]$ is defined by 
$$\{[F'_j]\in \B \mid F_j=F_j' \textup{ for } j \neq i\}.$$
Then the labeling of elements in $S$ agrees with the Dynkin diagram in Section \ref{sec:prelim}.

Let $\nn\in \g$ be contained in the minimal nilpotent orbit of $\g$, regarded as an endomorphism of $\k^{n+1}$, of Jordan type $(1^{n-1}2^1)$. 
According to \cite{vargas} (see also \cite{fung}) its corresponding Springer fiber is described as follows. For $1 \leq i \leq n$, let $X_i$ be the closed subvariety of $\B$ defined by 
$$X_i \colonequals \{[F_j] \in \B \mid \im \nn \subset F_{i} \subset \ker \nn\}.$$
Then it is easy to check that
$$X_i:Gr(i-1, n-1) \dashleftarrow \B(GL_i) \times \B(GL_{n+1-i}).$$
Thus $X_i$ are smooth varieties of dimension $\frac{1}{2}(n^2-n)$. Also $\Irr(\B_\nn) = \{X_1, X_2, \cdots, X_n\}$. For $i \neq j$, $\codim_{\B_\nn} X_i\cap X_j = |j-i|$ by direct calculation. Therefore $\Gamma_\nn$ is described as follows.
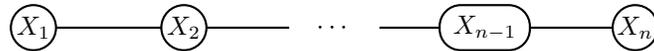
\begin{center}
\begin{adjustbox}{valign=c}
\begin{tikzpicture}[scale=1]
	\node at (0 cm,0) {$X_1$};
	\node at (2 cm,0) {$X_2$};
	\node at (4 cm,0) {$\cdots$};
	\node at (6 cm,0) {$X_{n-1}$};
	\node at (8 cm,0) {$X_n$};
	
    \draw[thick, rounded corners=8] (5.4, -0.3) rectangle (6.6, 0.3) {};

    \draw[thick] (0 cm,0) circle (3 mm);
    \draw[thick] (2 cm,0) circle (3 mm);
    \draw[thick] (8 cm,0) circle (3 mm);
    
    \draw[thick] (0.3, 0) -- (1.7,0);
    \draw[thick] (2.3, 0) -- (3.4,0);
    \draw[thick] (4.6, 0) -- (5.4,0);
    \draw[thick] (6.6, 0) -- (7.7,0);

\end{tikzpicture}
\end{adjustbox}
\end{center}
It is easy to check that $X_i$ is the union of lines of type $s_j$ for $i\neq j$, but not the union of lines of type $s_i$. It follows that $I_{X_i} = S -\{s_i\}.$

\subsection{Type $B$}\label{sec:typeB}
Let $G=SO_{2n+1}(\k)$ $(n \geq 2, \ch \k \neq 2)$. 
We regard $G$ as the group of automorphisms of determinant 1 acting on $\k^{2n+1}$ equipped with a non-degenerate symmetric bilinear form. Then we may identify $\B$ with the set of full isotropic flags in $\k^{2n+1}$, i.e.
$$\B=\{[F_i]=[F_0 \subset F_1 \subset \cdots \subset F_n]\mid \dim F_i = i, F_i \subset \k^{2n+1}\textup{ is isotropic} \}.$$ 

Let $s_i \in S$ be defined so that for any $[F_j] \in \B$ and $s_i \in S$, the line of type $s_i$ passing through $[F_j]$ is defined by 
$$\{[F'_j]\in \B \mid F_j=F_j' \textup{ for } j \neq i\}.$$
Then the labeling of elements in $S$ agrees with the Dynkin diagram in Section \ref{sec:prelim}.

Let $\nn \in \g$ be contained in the minimal nilpotent orbit of $\g$, which has Jordan type $(1^{2n-3}2^2)$.
We define $Y_i$ for $1 \leq i \leq n-1$ as follows.
$$Y_i \colonequals \{[F_j] \in \B \mid \im \nn \not\subset F_i, \textup{there exists a line } \vec{l} \subset \im \nn \textup{ such that }\vec{l}  \subset F_i \subset \ker \nn\}
$$
Also let $X_i$ be the closure of $Y_i$, i.e. $X_i \colonequals \overline{Y_i}$. Then 
\begin{gather*}
Y_1 : \P^1 \dashleftarrow \B(SO_{2n-1}),
\\Y_i : \P^1 \dashleftarrow OGr(i-2,2n-3) \dashleftarrow U_i \dashleftarrow \B(GL_i) \times \B(SO_{2n-2i+1}) \textup{ if } i \neq 1.
\end{gather*}
Here $U_i$ is a certain irreducible open subset of $Gr(i-2, i-1)$. Thus $Y_i$ are irreducible smooth varieties of dimension $n^2-2n+2$ and $X_i$ are irreducible of the same dimension. Also $\Irr(\B_\nn) = \{X_1, X_2, \cdots, X_{n-1}\}$. For $i \neq j$, $\codim_{\B_\nn} X_i\cap X_j = 1$ if and only if $|i-j|=1$ by \cite{dg:minimal}. Thus $\Gamma_\nn$ is described as follows.
\begin{center}
\begin{adjustbox}{valign=c}
\begin{tikzpicture}[scale=1]
	\node at (0 cm,0) {$X_1$};
	\node at (2 cm,0) {$X_2$};
	\node at (4 cm,0) {$\cdots$};
	\node at (6 cm,0) {$X_{n-1}$};
	
    \draw[thick, rounded corners=8] (5.4, -0.3) rectangle (6.6, 0.3) {};

    \draw[thick] (0 cm,0) circle (3 mm);
    \draw[thick] (2 cm,0) circle (3 mm);
    
    \draw[thick] (0.3, 0) -- (1.7,0);
    \draw[thick] (2.3, 0) -- (3.4,0);
    \draw[thick] (4.6, 0) -- (5.4,0);

\end{tikzpicture}
\end{adjustbox}
\end{center}
It is easy to check that $X_i$ is the union of lines of type $s_j$ for $i\neq j$, but not the union of lines of type $s_i$. It follows that $I_{X_i} = S -\{s_i\}.$

\begin{rmk} Here $Y_1 = X_1$, but $Y_i \subsetneq X_i$ for $i \neq 1$. Indeed, instead we may try to define
$$X_i \colonequals \{[F_j] \in \B \mid \textup{there exists a line } \vec{l} \subset \im \nn \textup{ such that } \vec{l}  \subset F_i \subset \ker \nn\}
$$
and argue as in Section \ref{sec:typeA}. However, then at some point one needs to examine the structure of the Grassmannian of isotropic $k'$-dimensional subspaces in $\k^{n'}$ for some $ k'\leq n'$, where $\k^{n'}$ is equipped with a ``degenerate'' symmetric bilinear form. Thus it is cumbersome to determine whether such $X_i$ are even irreducible.

Furthermore, in general not all the irreducible components are smooth; for example, if $G=SO_{7}$ we may show that 
$$\B_\nn \simeq \scvar{s(31232)} \cup \scvar{s(31231)}$$
 using the method of \cite{dg:minimal}. However, the Kazhdan-Lusztig polynomial $P_{id, s(31231)}(q) = q+1$, thus $\scvar{s(31231)}$ is not even rationally smooth. (cf. Lemma \ref{lem:smooth})
\end{rmk}

This time we let $\nn \in \g$ be contained in the minimal special nilpotent orbit of $\g$, which has Jordan type $(1^{2n-2}3^1)$.
For $1 \leq i \leq n$ define
$$X_i \colonequals \{ [F_j] \in \B \mid \im \nn^2 \subset F_i \subset \ker \nn\}.$$
Then
$$X_i : OGr(i-1,2n-2) \dashleftarrow \B(GL_i) \times \B(SO_{2n-2i+1}).$$
Thus $X_i$ are smooth varieties of dimension $n^2-2n+1$, and irreducible except when $i=n$, in which case $X_n$ is the union of two irreducible components of $\B_\nn$, denoted $X_n', X_n''$. Then $\Irr(\B_\nn) = \{X_1, X_2, \cdots, X_{n-1}, X_n', X_n''\}$. For $i \neq j < n$, $\codim_{\B_\nn} X_i \cap X_j = |i-j|$, $\codim_{\B_\nn} X_i \cap X_n' = \codim_{\B_\nn} X_i \cap X_n'' = n-i,$ and $X_n' \cap X_n'' = \emptyset$, which can be verified by direct calculation. Thus $\Gamma_\nn$ is described as follows.
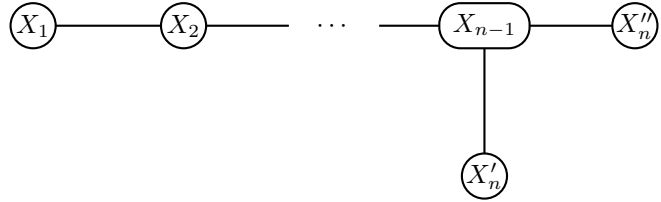
\begin{center}
\begin{adjustbox}{valign=c}
\begin{tikzpicture}[scale=1]
	\node at (0 cm,0) {$X_1$};
	\node at (2 cm,0) {$X_2$};
	\node at (4 cm,0) {$\cdots$};
	\node at (6 cm,0) {$X_{n-1}$};
	\node at (6 cm,-2) {$X_n'$};
	\node at (8 cm,0) {$X_n''$};
    \draw[thick, rounded corners=8] (5.4, -0.3) rectangle (6.6, 0.3) {};
    
    \draw[thick] (0 cm,0) circle (3 mm);
    \draw[thick] (2 cm,0) circle (3 mm);
    \draw[thick] (8 cm,0) circle (3 mm);
    \draw[thick] (6 cm,-2cm) circle (3 mm);
    
    \draw[thick] (0.3, 0) -- (1.7,0);
    \draw[thick] (2.3, 0) -- (3.4,0);
    \draw[thick] (4.6, 0) -- (5.4,0);
    \draw[thick] (6.6, 0) -- (7.7,0);
    \draw[thick] (6, -0.3) -- (6,-1.7);
    
\end{tikzpicture}
\end{adjustbox}
\end{center}
It is easily checked that $I_{X_i} = S- \{s_i\}$, $I_{X_n'} = I_{X_n''} = S-\{s_n\}$. The action of the nontrivial element in $C(\nn) \simeq \Z/2$ is given by the nontrivial automorphism of $\Gamma_\nn$ which permutes $X_n'$ and $X_n''$.

\subsection{Type $C$}\label{sec:typeC}
Let $G=Sp_{2n}(\k)$ $(n\geq 2, \ch \k \neq 2)$. 
We regard $G$ as the automorphism group of $\k^{2n}$ equipped with a non-degenerate skew-symmetric bilinear form. Then we may identify $\B$ with the set of full isotropic flags in $\k^{2n}$, i.e.
$$\B=\{[F_i]=[F_0 \subset F_1 \subset \cdots \subset F_n]\mid \dim F_i = i, F_i \subset \k^{2n}\textup{ is isotropic} \}.$$ 

Let $s_i\in S$ be chosen so that for any $[F_j] \in \B$ and $s_i \in S$, the line of type $s_i$ passing through $[F_j]$ is defined by 
$$\{[F'_j]\in \B \mid F_j=F_j' \textup{ for } j \neq i\}.$$
Then the labeling of elements in $S$ agrees with the Dynkin diagram in Section \ref{sec:prelim}.
%

Let $\nn \in \g$ be contained in the minimal nilpotent orbit of $\g$, which has Jordan type $(1^{2n-2}2^1)$. Then
%
$$\B_\nn = \{ [F_j] \in \B \mid \im \nn \subset F_n \subset \ker \nn\},$$
which is irreducible. Also we have
$$\B_\nn : SGr(n-1,2n-2) \dashleftarrow \B(GL_{n}),$$
thus $\B_\nn$ is smooth of dimension $n^2-n$. It follows that $\Gamma_\nn$ has only one vertex, i.e. 
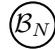
\begin{center}
\begin{adjustbox}{valign=c}
\begin{tikzpicture}[scale=1]
	\node at (0 cm,0) {$\B_\nn$};
	\draw[thick] (0 cm,0) circle (3 mm);
\end{tikzpicture}
\end{adjustbox}
\end{center}
It is easy to see that $I_{\B_\nn} = S-\{s_n\}$.

This time let $\nn \in \g$ be contained in the minimal special nilpotent orbit of $\g$, which has Jordan type $(1^{2n-4}2^2)$. Then $\nn\cdot\{ x \in \k^{2n} \mid \br{x, \nn x}=0\}$
consists of two lines, say $\vec{l},  \vec{l}' \subset \k^{2n}$, which span the image of $N$.
%
Now we define
\begin{gather*}
X_i \colonequals \{[F_j]\in\B \mid \vec{l} \subset F_i \subset \ker \nn\}
\\X'_i \colonequals \{[F_j]\in\B \mid \vec{l}' \subset F_i \subset \ker \nn\}
\\X_n = X'_n \colonequals \{[F_j]\in\B \mid \im \nn \subset F_n \subset \ker \nn\}
\end{gather*}
Then
$$X_i, X_i' : SGr(i-1, 2n-3) \dashleftarrow \B(GL_i) \times \B(Sp_{2n-2i}).$$
(Here $SGr(i-1,2n-3)$ denotes the ``odd'' symplectic Grassmannian defined in \cite{mihai}.) Thus $X_i, X_i'$ are smooth of dimension $n^2-2n+1$. Also, $\Irr(\B_\nn) = \{X_1, \cdots, X_n=X_n'', X_1', \cdots, X_{n-1}'\}$. We have $\codim_{\B_\nn} X_i \cap X_j = |i-j|$ and $\codim_{\B_\nn} X_i \cap X_j' = 2n-i-j$ by direct calculation, thus $\Gamma_\nn$ is as follows.
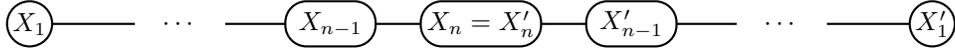
\begin{center}
\begin{adjustbox}{valign=c}
\begin{tikzpicture}[scale=1]
	\node at (2 cm,0) {$X_1$};
	\node at (4 cm,0) {$\cdots$};
	\node at (6 cm,0) {$X_{n-1}$};
	\node at (8 cm,0) {$X_n=X_n'$};
	\node at (10 cm,0) {$X_{n-1}'$};
	\node at (12 cm,0) {$\cdots$};
	\node at (14 cm,0) {$X_1'$};
	
    \draw[thick, rounded corners=8] (5.4, -0.3) rectangle (6.6, 0.3) {};
    \draw[thick, rounded corners=8] (7.2, -0.3) rectangle (8.8, 0.3) {};
    \draw[thick, rounded corners=8] (9.4, -0.3) rectangle (10.6, 0.3) {};    

    \draw[thick] (2 cm,0) circle (3 mm);
    \draw[thick] (14 cm,0) circle (3 mm);
    
    \draw[thick] (2.3, 0) -- (3.4,0);
    \draw[thick] (4.6, 0) -- (5.4,0);
    \draw[thick] (6.6, 0) -- (7.2,0);
    \draw[thick] (8.8, 0) -- (9.4,0);
    \draw[thick] (10.6, 0) -- (11.4,0);
    \draw[thick] (12.6, 0) -- (13.7,0);    

\end{tikzpicture}
\end{adjustbox}
\end{center}
It is easy to check that $I_{X_i}=I_{X_i'}= S-\{s_i\}.$ The nontrivial element of $C(\nn) \simeq \Z/2$ acts on $\Gamma_\nn$ by permuting $X_i$ and $X_i'$.

\subsection{Type $D$} \label{sec:typeD}
Let $G=SO_{2n}(\k)$ $(n\geq 3, \ch \k \neq 2)$. 
We regard $G$ as the group of automorphisms with determinant 1 acting on $\k^{2n}$ equipped with a non-degenerate symmetric bilinear form. Then we may identify $\B$ with the set of certain isotropic flags in $\k^{2n}$, i.e.
$$\B=\{[F_i]=[F_0 \subset F_1 \subset \cdots\subset F_{n-1}]\mid \dim F_i = i, F_i \subset \k^{2n}\textup{ is isotropic} \}.$$ 

We let $s_i \in S$ such that they satisfy the following conditions. For any $[F_j] \in \B$ and $s_i \in S$ such that $1\leq i\leq n-2$, the line of type $s_i$ passing through $[F_j]$ is defined by 
$$\{[F'_j]\in \B \mid F_j=F_j' \textup{ for } j \neq i\}.$$
If $i=n-1$ or $n$, first note that there exist exactly two Lagrangian subspaces $F_{n,1},F_{n,2}\subset \k^{2n}$ which contain $F_{n-1}$. They can be labeled in a way that for any two flags $[F_j]$ and $[F'_j]$ and $g \in G$ such that $g\cdot [F_j] = [F'_j]$, we have $g\cdot F_{n,1} = F'_{n,1}$ and $g\cdot F_{n,2} = F'_{n,2}$.
Then the line of type $s_{n-1}$, $s_n$, respectively, passing through $[F_j]$ is defined by
\begin{gather*}
\{[F_j'] \in \B \mid F_j=F_j' \textup{ for } j \neq n-1, F_{n,1}=F_{n,1}'\}, \quad
\{[F_j'] \in \B \mid F_j=F_j' \textup{ for } j \neq n-1, F_{n,2}=F_{n,2}'\},
\end{gather*}
respectively. Note that such labeling of $s_i \in S$ agrees with the Dynkin diagram in Section \ref{sec:prelim}.

Let $\nn \in \g$ to be contained in the minimal nilpotent orbit of $\g$, which has Jordan type $(1^{2n-4}2^2)$. Then
For $1 \leq i \leq n-1$ we define
$$Y_i \colonequals \{[F_j] \in \B \mid \im \nn \not\subset F_i, \textup{there exists a line } \vec{l} \subset\im \nn \textup{ such that } \vec{l}  \subset F_i \subset \ker \nn\}
$$
Also let $X_i$ be the closure of $X_i$, i.e. $X_i \colonequals \overline{Y_i}$. Then 
\begin{gather*}
Y_1 : \P^1 \dashleftarrow \B(SO_{2n-2}),
\\Y_i : \P^1 \dashleftarrow OGr(i-1,2n-4) \dashleftarrow U_i \dashleftarrow \B(GL_i) \times \B(SO_{2n-2i}) \textup{ if } i \neq 1.
\end{gather*}
Here $U_i$ is a certain irreducible open subset of $Gr(i-1,i)$. Thus $Y_i$ are smooth varieties of dimension $n^2-3n+3$ and $X_i$ are of the same dimension. They are also irreducible except when $i=n-1$, in which case $X_{n-1}$ is a union of two irreducible components of $\B_\nn$, denoted $X_{n-1}'$ and $X_{n-1}''$. Then $\Irr \B_\nn = \{X_1, \cdots, X_{n-2}, X_{n-1}', X_{n-1}''\}$. For $i, j \leq n-2$, by \cite{dg:minimal} $\codim_{\B_\nn} X_i \cap X_j=1$ if and only if $|i-j|=1$. Also, $\codim_{\B_\nn} X_j \cap X_{n-1}'=1$ and/or $\codim_{\B_\nn} X_j \cap X_{n-1}''=1$ if and only if $j=n-2$. However, $\codim_{\B_\nn} X_{n-1}' \cap X_{n-1}'' >1$. Therefore $\Gamma_\nn$ is as follows.
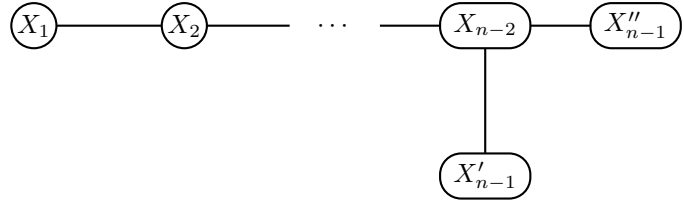
\begin{center}
\begin{adjustbox}{valign=c}
\begin{tikzpicture}[scale=1]
	\node at (0 cm,0) {$X_1$};
	\node at (2 cm,0) {$X_2$};
	\node at (4 cm,0) {$\cdots$};
	\node at (6 cm,0) {$X_{n-2}$};
	\node at (6 cm,-2) {$X_{n-1}'$};
	\node at (8 cm,0) {$X_{n-1}''$};
    \draw[thick, rounded corners=8] (5.4, -0.3) rectangle (6.6, 0.3) {};
    \draw[thick, rounded corners=8] (7.4, -0.3) rectangle (8.6, 0.3) {};
    \draw[thick, rounded corners=8] (5.4, -2.3) rectangle (6.6, -1.7) {};
    
    \draw[thick] (0 cm,0) circle (3 mm);
    \draw[thick] (2 cm,0) circle (3 mm);
    
    \draw[thick] (0.3, 0) -- (1.7,0);
    \draw[thick] (2.3, 0) -- (3.4,0);
    \draw[thick] (4.6, 0) -- (5.4,0);
    \draw[thick] (6.6, 0) -- (7.4,0);
    \draw[thick] (6, -0.3) -- (6,-1.7);
    
\end{tikzpicture}
\end{adjustbox}
\end{center}
It is easy to check that $I_{X_{i}} = S- \{s_i\}$ for $1\leq i \leq n-2$, and $\{I_{X_{n-1}'}, I_{X_{n-1}''}\} = \{S-\{s_{n-1}\}, S-\{s_n\}\}$.

\begin{rmk} Here we have a similar issue as in the remark of Section \ref{sec:typeB}, thus the argument here is also similar to it. Furthermore, in general $X_i$ are not smooth; for example, if $G=SO_{8}$ we may show that 
$$\B_\nn \simeq \scvar{s(3124232)} \cup \scvar{s(4123121)} \cup \scvar{s(3124231)} \cup \scvar{s(3124121)}$$
 using the method of \cite{dg:minimal}. However, the Kazhdan-Lusztig polynomial $P_{id, s(3124231)}(q) = 2q+1$, thus $\scvar{s(3124231)}$ is not even rationally smooth. (cf. Lemma \ref{lem:smooth})
\end{rmk}

\section{Exceptional types}
For exceptional types, it is more difficult to describe each irreducible component of $\B_\nn$ as explicitly as classical types. Instead, in this section we choose a representative $\nn \in \g$ carefully and study the intersection of $\B_\nn$ with each (open) Schubert cell. From now on, we let $\lambda \in \Phi^+$ be the highest root and $\mu \in \Phi^+$ be the highest root among short roots if $\Phi$ have roots of two different lengths. We start with the following observation. 

\begin{lem}\label{lem:choice}  If $\nn \in \g$ is a nontrivial element in $X_\lambda \subset \n$, then $\nn$ is contained in the minimal nilpotent orbit of $\g$. If $G$ is not of type $G_2$ and $\nn \in \g$ is a nontrivial element in $X_\mu \subset \n$, then $\nn$ is contained in the minimal special nilpotent orbit of $\g$.
\end{lem}
\begin{proof} This can be checked case-by-case: see \cite{dg:minimal} for minimal nilpotent ones.
\end{proof}

Thus we may assume either $\nn \in X_\lambda$ or $\nn \in X_\mu$. Then the following lemma is our main tool.
\begin{lem}\label{lem:explicit} We keep the assumptions and notations in Lemma \ref{lem:choice}.
\begin{enumerate}
\item If $\nn \in \g$ is a nontrivial element in $X_\lambda \subset \n$, then $\B_\nn \cap \S_w \neq \emptyset$ if and only if $w^{-1}(\lambda)>0$. If so, then $\B_\nn \cap \S_w= \S_w$.
\item If $\nn \in \g$ is a nontrivial element in $X_\mu \subset \n$, then $\B_\nn \cap \S_w \neq \emptyset$ if and only if $w^{-1}(\mu)>0$. If so, then $\B_\nn \cap \S_w = 
\left(\prod_{\alpha} U_\alpha\right) w B/B$, where the product is over all the roots $\alpha \in R^+$ such that 
$$w^{-1}(\alpha)<0, \text{ and either } \alpha+\mu \notin \Phi \text{ or } w^{-1}(\alpha+\mu)>0.$$
\end{enumerate}
\end{lem}
\begin{proof}
This is basically \cite[7.10--7.14]{springer:trig} with a minor correction, which fixes an error in the proof of \cite[Proposition 7.11]{springer:trig}.
\end{proof}

Also to check the smoothness of irreducible components we use the following lemma.
\begin{lem}\label{lem:smooth} For $w\in W$, the Schubert variety $\scvar{w}$ is rationally smooth if and only if the Kazhdan-Lusztig polynomial $P_{id, w}(q)$ is equal to 1 if and only if the Poincar{\'e} polynomial of $\scvar{w}$ is palindromic. Furthermore, if $G$ is of simply-laced type, $\scvar{w}$ is rationally smooth if and only if it is smooth.
\end{lem}
\begin{proof} The first part is a consequence of \cite{kl:schubert} and \cite{carrell:deodhar}, and the second part is proved in \cite{ck:peterson}.
\end{proof}

\subsection{Type $E_6$} If we let $\nn \in \g$ be a nontrivial element in $X_\lambda \subset \n$, then $\nn$ is contained in the minimal nilpotent orbit which is also special. In this case we use Lemma \ref{lem:explicit} to calculate $\Gamma_\nn$ directly, which is also described in \cite{dg:minimal}. By direct calculation, we obtain the following.

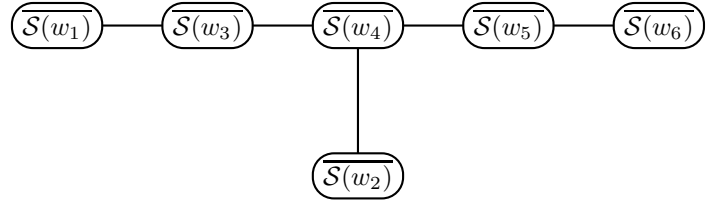
\begin{center}
\begin{tikzpicture}[scale=1]
	\node at (0 cm,0) {$\scvar{w_1}$};
	\node at (2 cm,0) {$\scvar{w_3}$};
	\node at (4 cm,0) {$\scvar{w_4}$};
	\node at (6 cm,0) {$\scvar{w_5}$};
	\node at (8 cm,0) {$\scvar{w_6}$};
	\node at (4 cm,-2cm) {$\scvar{w_2}$};

    \draw[thick, rounded corners=8] (-0.6, -0.3) rectangle (0.6, 0.3) {};
    \draw[thick, rounded corners=8] (1.4, -0.3) rectangle (2.6, 0.3) {};
    \draw[thick, rounded corners=8] (3.4, -0.3) rectangle (4.6, 0.3) {};
    \draw[thick, rounded corners=8] (5.4, -0.3) rectangle (6.6, 0.3) {};
    \draw[thick, rounded corners=8] (7.4, -0.3) rectangle (8.6, 0.3) {};
    \draw[thick, rounded corners=8] (3.4, -2.3) rectangle (4.6, -1.7) {};
    
    \draw[thick] (0.6, 0) -- (1.4,0);
    \draw[thick] (2.6, 0) -- (3.4,0);
    \draw[thick] (4.6, 0) -- (5.4,0);
    \draw[thick] (6.6, 0) -- (7.4,0);
    \draw[thick] (4, -0.3) -- (4,-1.7);
\end{tikzpicture}
\end{center}
Here each $w_i$ is as follows.
\begin{gather*}
w_1=s(5645341324565413245432432)
\\w_2=s(5645341324565413245134131)
\\w_3=s(5645341324565413245432421)
\\w_4=s(5645341324565413245341321)
\\w_5=s(5645341324562453413241321)
\\w_6=s(1345624534132453413241321)
\end{gather*}
Also, $I_{i} = S-\{\alpha_i\}$. The Kazhdan-Lusztig polynomial $P_{id, w_i}(q)$ for each $w_i$ is as follows.
\begin{align*}
&P_{id,w_1}(q)=1, &&P_{id,w_2}(q)=q^3+q^2+1, &&P_{id,w_3}(q)=q^2+q+1
\\&P_{id,w_4}(q)=q^3+2q^2+2q+1, &&P_{id,w_5}(q)=q^2+q+1, &&P_{id,w_6}(q)=1
\end{align*}
Thus by Lemma \ref{lem:smooth} $\scvar{w_1}$ and $\scvar{w_6}$ are smooth, but others are not.

\subsection{Type $E_7$} We argue the same as in type $E_6$ and obtain the following.
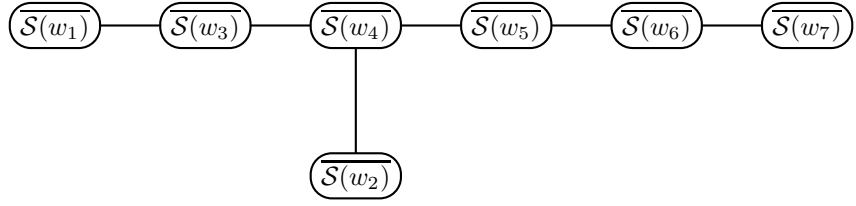
\begin{center}
\begin{tikzpicture}[scale=1]
	\node at (0 cm,0) {$\scvar{w_1}$};
	\node at (2 cm,0) {$\scvar{w_3}$};
	\node at (4 cm,0) {$\scvar{w_4}$};
	\node at (6 cm,0) {$\scvar{w_5}$};
	\node at (8 cm,0) {$\scvar{w_6}$};
	\node at (10 cm,0) {$\scvar{w_7}$};
	\node at (4 cm,-2cm) {$\scvar{w_2}$};

    \draw[thick, rounded corners=8] (-0.6, -0.3) rectangle (0.6, 0.3) {};
    \draw[thick, rounded corners=8] (1.4, -0.3) rectangle (2.6, 0.3) {};
    \draw[thick, rounded corners=8] (3.4, -0.3) rectangle (4.6, 0.3) {};
    \draw[thick, rounded corners=8] (5.4, -0.3) rectangle (6.6, 0.3) {};
    \draw[thick, rounded corners=8] (7.4, -0.3) rectangle (8.6, 0.3) {};
    \draw[thick, rounded corners=8] (9.4, -0.3) rectangle (10.6, 0.3) {};
    \draw[thick, rounded corners=8] (3.4, -2.3) rectangle (4.6, -1.7) {};
    
    \draw[thick] (0.6, 0) -- (1.4,0);
    \draw[thick] (2.6, 0) -- (3.4,0);
    \draw[thick] (4.6, 0) -- (5.4,0);
    \draw[thick] (6.6, 0) -- (7.4,0);
    \draw[thick] (8.6, 0) -- (9.4,0);
    \draw[thick] (4, -0.3) -- (4,-1.7);
\end{tikzpicture}
\end{center}
Here each $w_i$ is as follows.
\begin{gather*}
w_1=s(2456734562453413245675645341324565413245432432)
\\w_2=s(2456734562453413245675645341324565413245134131)
\\w_3=s(2456734562453413245675645341324565413245432421)
\\w_4=s(2456734562453413245675645341324565413245341321)
\\w_5=s(2456734562453413245675645341324562453413241321)
\\w_6=s(2456734562453413245671345624534132453413241321)
\\w_7=s(7654324567134562453413245624534132453413241321)
\end{gather*}
Also $I_{w_i} = S-\{\alpha_i\}$. The Kazhdan-Lusztig polynomial $P_{id, w_i}(q)$ for each $w_i$ is as follows.
\begin{align*}
&P_{id,w_1}(q)=q^5+q^3+1, &&P_{id,w_2}(q)=q^3+q^2+1, 
\\&P_{id,w_3}(q)=q^5+q^4+q^3+q^2+q+1, &&P_{id,w_4}(q)=q^5+q^4+3q^3+2q^2+2q+1,
\\&P_{id,w_5}(q)=q^4+q^3+2q^2+q+1, &&P_{id,w_6}(q)=q^3+q+1,
\\&P_{id,w_7}(q)=1
\end{align*}
Thus by Lemma \ref{lem:smooth} only $\scvar{w_7}$ is smooth.

\subsection{Type $E_8$} We do the same procedure as in type $E_6, E_7$ and obtain the following.

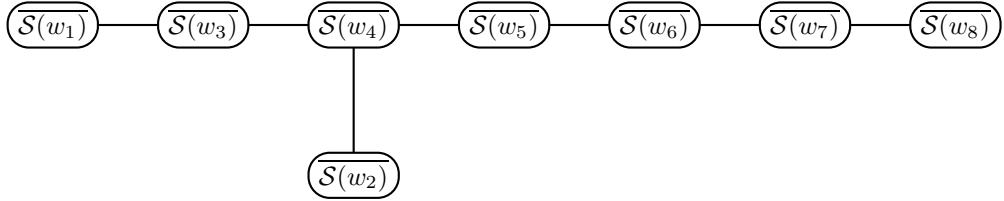
\begin{center}
\begin{tikzpicture}[scale=1]
	\node at (0 cm,0) {$\scvar{w_1}$};
	\node at (2 cm,0) {$\scvar{w_3}$};
	\node at (4 cm,0) {$\scvar{w_4}$};
	\node at (6 cm,0) {$\scvar{w_5}$};
	\node at (8 cm,0) {$\scvar{w_6}$};
	\node at (10 cm,0) {$\scvar{w_7}$};
	\node at (12 cm,0) {$\scvar{w_8}$};
	\node at (4 cm,-2cm) {$\scvar{w_2}$};

    \draw[thick, rounded corners=8] (-0.6, -0.3) rectangle (0.6, 0.3) {};
    \draw[thick, rounded corners=8] (1.4, -0.3) rectangle (2.6, 0.3) {};
    \draw[thick, rounded corners=8] (3.4, -0.3) rectangle (4.6, 0.3) {};
    \draw[thick, rounded corners=8] (5.4, -0.3) rectangle (6.6, 0.3) {};
    \draw[thick, rounded corners=8] (7.4, -0.3) rectangle (8.6, 0.3) {};
    \draw[thick, rounded corners=8] (9.4, -0.3) rectangle (10.6, 0.3) {};
    \draw[thick, rounded corners=8] (11.4, -0.3) rectangle (12.6, 0.3) {};
    \draw[thick, rounded corners=8] (3.4, -2.3) rectangle (4.6, -1.7) {};
    
    \draw[thick] (0.6, 0) -- (1.4,0);
    \draw[thick] (2.6, 0) -- (3.4,0);
    \draw[thick] (4.6, 0) -- (5.4,0);
    \draw[thick] (6.6, 0) -- (7.4,0);
    \draw[thick] (8.6, 0) -- (9.4,0);
    \draw[thick] (10.6, 0) -- (11.4,0);
    \draw[thick] (4, -0.3) -- (4,-1.7);
\end{tikzpicture}
\end{center}
Here each $w_i$ is as follows.
\begin{align*}
w_1 = s(&134562453413245678654324567134562453413245678
\\&2456734562453413245675645341324565413245432432)
\\w_2=s(&134562453413245678654324567134562453413245678
\\&2456734562453413245675645341324565413245134131)
\\w_3=s(&134562453413245678654324567134562453413245678
\\&2456734562453413245675645341324565413245432421)
\\w_4=s(&134562453413245678654324567134562453413245678
\\&2456734562453413245675645341324565413245341321)
\\w_5=s(&134562453413245678654324567134562453413245678
\\&2456734562453413245675645341324562453413241321)
\\w_6=s(&134562453413245678654324567134562453413245678
\\&2456734562453413245671345624534132453413241321)
\\w_7=s(&134562453413245678654324567134562453413245678
\\&7654324567134562453413245624534132453413241321)
\\w_8=s(&765432456713456245341324567876543245671345624
\\&5341324567134562453413245624534132453413241321)
\end{align*}
Also $I_{w_i} = S-\{\alpha_i\}$.  Instead of Kazhdan-Lusztig polynomials $P_{id, w_i}$, we calculate the Poincar{\'e} polynomial $\mathbf{P}_{(id, w_i)}$ for each Bruhat inverval $(id, w_i)$ as follows.\footnote{The author thanks David Vogan for assistance to this calculation.}
\begin{alignat*}{6}
\mathbf{P}_{(id,w_1)}(q)=&q^{91}+8q^{90}+35q^{89}+113q^{88}+ &\cdots&& + 112q^3+35q^2+8q+1,
\\\mathbf{P}_{(id,w_2)}(q)=&q^{91}+8q^{90}+36q^{89}+ &\cdots&& +35q^2+8q+1,
\\\mathbf{P}_{(id,w_3)}(q)=&q^{91}+9q^{90}+ &\cdots&& +8q+1,
\\\mathbf{P}_{(id,w_4)}(q)=&q^{91}+10q^{90}+ &\cdots&&+8q+1,
\\\mathbf{P}_{(id,w_5)}(q)=&q^{91}+9q^{90}+ &\cdots&& +8q+1,
\\\mathbf{P}_{(id,w_6)}(q)=&q^{91}+9q^{90}+ &\cdots&& +8q+1,
\\\mathbf{P}_{(id,w_7)}(q)=&q^{91}+9q^{90}+ &\cdots&& +8q+1,
\end{alignat*}
\begin{align*}
\mathbf{P}_{(id,w_8)}(q)=&q^{91}+8q^{90}+35q^{89}+112q^{88}+294q^{87} +673 q^{86}+
\\&\qquad \qquad \cdots+672q^5+294q^4+ 112q^3+35q^2+8q+1.
\end{align*}
Thus, $\mathbf{P}_{(id,w_i)}$ are not palindromic for all $1\leq i \leq 8$, which implies that no $\scvar{w_i}$ is smooth by Lemma \ref{lem:smooth}.\footnote{It is indeed extremely time-consuming to calculate the Kazhdan-Lusztig polynomials $P_{id, w_i}$.}

\subsection{Type $F_4$}
Let $G$ be a simple group of type $F_4$ and assume $\nn \in \g$ is a non-trivial element in $X_\lambda \subset \n$. Then $\nn$ is contained in the minimal nilpotent orbit. As in type $E$, $\Gamma_\nn$ is described as
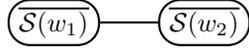
\begin{center}
\begin{tikzpicture}[scale=1]
	\node at (0 cm,0) {$\scvar{w_1}$};
	\node at (2 cm,0) {$\scvar{w_2}$};
    \draw[thick, rounded corners=8] (-0.6, -0.3) rectangle (0.6, 0.3) {};
    \draw[thick, rounded corners=8] (1.4, -0.3) rectangle (2.6, 0.3) {};
    \draw[thick] (0.6, 0) -- (1.4,0);
\end{tikzpicture}
\end{center}
Here each $w_i$ is as follows.
\begin{gather*}
w_1=s(3234323123431232)
\\w_2=s(3234323123431231)
\end{gather*}
Also $I_{w_i} = S-\{\alpha_i\}$. The Kazhdan-Lusztig polynomial $P_{id, w_i}$ for each $w_i$ is as follows.
\begin{align*}
&P_{id,w_1}(q)=q^3+1, &&P_{id,w_2}(q)=q^2+q+1 
\end{align*}
Thus both components are not even rationally smooth.

Now we assume $\nn \in \g$ is a nontrivial element in $X_\mu \subset \n$. Then $\nn$ is contained in the minimal special nilpotent orbit. As the justification for this case is lengthy and complicated, we state the result here and give its proof in the appendix. $\Gamma_\nn$ is described as follows. 

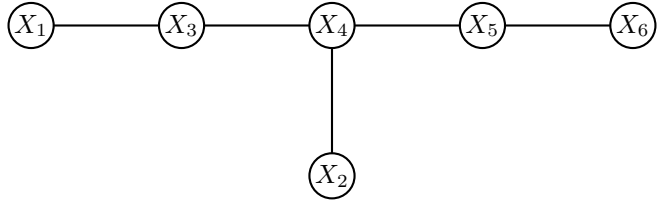
\begin{center}
\begin{tikzpicture}[scale=1]
	\node at (0 cm,0) {$X_4$};
	\node at (2 cm,0) {$X_5$};
	\node at (4 cm,0) {$X_6$};
	\node at (-4 cm,0) {$X_1$};
	\node at (-2 cm,0) {$X_3$};
	\node at (0 ,-2cm) {$X_2$};

    \draw[thick] (0 cm,0) circle (3 mm);
    \draw[thick] (2 cm,0) circle (3 mm);
    \draw[thick] (4 cm,0) circle (3 mm);
    \draw[thick] (-4 cm,0) circle (3 mm);
    \draw[thick] (-2 cm,0) circle (3 mm);
    \draw[thick] (0,-2cm) circle (3 mm);    
    
    \draw[thick] (0.3, 0) -- (1.7,0);
    \draw[thick] (2.3, 0) -- (3.7,0);
    \draw[thick] (-0.3, 0) -- (-1.7,0);
    \draw[thick] (-2.3, 0) -- (-3.7,0);
    \draw[thick] (0, -0.3) -- (0,-1.7);
\end{tikzpicture}
\end{center}
Also $I_{X_1} =I_{X_6}=\{s_1, s_2, s_3\}, I_{X_2} = \{s_2, s_3, s_4\}, I_{X_3}=I_{X_5} = \{s_1, s_2, s_4\}, I_{X_4} = \{s_1, s_3, s_4\}$. The action of the nontrivial element in $C(\nn) \simeq \Z/2$ is given by the nontrivial automorphism on $\Gamma_\nn$ which permutes $X_1$ and $X_6$, and $X_3$ and $X_5$.
\begin{rmk} It is likely that not every irreducible component is smooth in this case.
\end{rmk}

\subsection{Type $G_2$} If $\nn \in \g$ is a nontrivial element in $X_\lambda \subset \n$, then $\nn$ is contained in the minimal nilpotent orbit. Then $\B_\nn = \scvar{s(121)}$ using the method of \cite{dg:minimal}, thus $\Gamma_\nn$ is
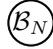
\begin{center}
\begin{adjustbox}{valign=c}
\begin{tikzpicture}[scale=1]
	\node at (0 cm,0) {$\B_\nn$};
	\draw[thick] (0 cm,0) circle (3 mm);
\end{tikzpicture}
\end{adjustbox}
\end{center}
Also, $I_{\B_\nn}= \{s_1\}$. Note that $\scvar{s(121)}$ is not smooth by \cite[Theorem 2.4]{bp:schubert}, even though it is rationally smooth because the Kazhdan-Lusztig polynomial $P_{id, s(121)}(q)$ is 1.

For type $G_2$, the minimal special nilpotent orbit is the subregular one. Thus by result of \cite[3.10]{steinberg:book}, $\B_\nn$ is a Dynkin curve and $\Gamma_\nn$ is as follows.
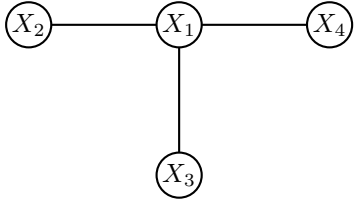
\begin{center}
\begin{tikzpicture}[scale=1]
	\node at (0 cm,0) {$X_2$};
	\node at (2 cm,0) {$X_1$};
	\node at (4 cm,0) {$X_4$};
	\node at (2 cm,-2cm) {$X_3$};
	\draw[thick] (0 cm,0) circle (3mm);
	\draw[thick] (2 cm,0) circle (3mm);
	\draw[thick] (4 cm,0) circle (3mm);
	\draw[thick] (2 cm,-2cm) circle (3mm);
    \draw[thick] (0.3, 0) -- (1.7,0);
    \draw[thick] (2.3, 0) -- (3.7,0);
    \draw[thick] (2, -0.3) -- (2,-1.7);
\end{tikzpicture}
\end{center}
Also $I_{X_1} = \{\alpha_2\}, I_{X_2}=I_{X_3}=I_{X_4}=\{\alpha_1\}$. Every irreducible component is $\P^1$, hence smooth. The action of $C(\nn) \simeq S_3$ stablizes $X_1$ but permutes $X_2, X_3,$ and $X_4$ faithfully.

\section{Relations with folding of Lie algebras}
This section is motivated by the question of Humphreys in \cite{humphreys:mof} and the comment by Paul Levy therein. Assume $G$ is simple of simply-laced type, and $\sigma$ is an automorphism on $G$ which gives a nontrivial action on the Dynkin diagram of $G$. We list some of possible choices of $(G, \sigma)$. (Here $|\sigma|$ is the order of $\sigma$.)
\begin{enumerate}[(1)]
\item $G$ is of type $A_{2n-1}, |\sigma|=2$: $G^\sigma$ is of type $C_n$
\item $G$ is of type $D_{n+1}, |\sigma|=2$: $G^\sigma$ is of type $B_n$
\item $G$ is of type $D_{4}, |\sigma|=3$: $G^\sigma$ is of type $G_2$
\item $G$ is of type $E_6, |\sigma|=2$: $G^\sigma$ is of type $F_4$
\end{enumerate}

We describe how these automorphisms are related to the results we obtained so far. Assume that $\ch \k$ is good for $G$ and $G^\sigma$. Thus in particular $|\sigma| \in \k^*$. We may choose $T, B\subset G$ such that they are stable under $\sigma$, and each $T^\sigma, B^\sigma \subset G^\sigma$ is a maximal torus and a Borel subgroup of $G^\sigma$, respectively. (cf. \cite[Chapter 8]{steinberg:endomorphism}) Then there exists a natural embedding $G^\sigma/B^\sigma \hookrightarrow G/B,$ i.e. $\B(G^\sigma) \hookrightarrow \B(G)=\B$. Let $\nn \in \g$ be contained in the minimal nilpotent orbit of $\g$ and denote also by $\sigma$ its induced action on $\g$. Then $\nn'\colonequals \sum_{i=0}^{|\sigma|-1 }\sigma^i(\nn) \in \g^\sigma$. We assume that $\nn$ is chosen generically, then by \cite{bk:folding} $\nn'$ is contained in the minimal special nilpotent orbit of $\g^\sigma$. 

We claim that $\dim \B-\dim \B_\nn=  \dim \B(G^\sigma)- \dim \B(G^\sigma)_{\nn'}.$ (This can also be checked case-by-case.) We denote by $\O(\nn), \O(\nn')$ the orbit of $\nn$ in $\g$ and that of $\nn'$ in $\g^\sigma$, respectively. By \cite{bk:folding}, there is a $\g^\sigma$-equivariant projection $\pi: \g \rightarrow \g^\sigma$ which induces a finite morphism $\overline{\O(\nn)} \rightarrow \overline{\O(\nn')}$. In particular, $\dim \O(\nn) = \dim \O(\nn')$. As for any nilpotent $\tilde{\nn} \in \g$
$$\dim \B_{\tilde{\nn}} = \dim U - \frac{1}{2} \dim \O(\tilde{\nn}), \textup{ i.e. } \dim \B - \dim \B_{\tilde{\nn}} = \frac{1}{2} \dim \O(\tilde{\nn}),$$
the result follows.

Note that $\B_\nn \cap \B(G^\sigma) \subset \B(G^\sigma)_{\nn'}.$ Thus we have
$$\dim \B_\nn \cap \B(G^\sigma) \leq \dim \B(G^\sigma)_{\nn'}= \dim \B_\nn+ \dim \B(G^\sigma)-\dim \B \leq \dim \B_\nn \cap \B(G^\sigma),$$
where the equality in the middle is obtained from above and the latter inequality follows from the relation of codimensions of varieties and their intersection. Thus $\dim \B_\nn \cap \B(G^\sigma)=\dim  \B(G^\sigma)_{\nn'}$
and $\B_\nn \cap \B(G^\sigma)$ contains some irreducible components of $ \B(G^\sigma)_{\nn'}$.

Note that $\Gamma_{\nn'}$ and $\Gamma_\nn$ have the same shape by case-by-case observation and in particular $|\Irr \B_\nn| = |\Irr  \B(G^\sigma)_{\nn'}|$. Thus if we can show that
\begin{equation}\tag{$*$}\label{eq:int}
\textup{for any } X \in \Irr(\B_\nn), \quad X \cap \B(G^\sigma) \neq \emptyset,
\end{equation}
it follows that $\B_\nn \cap \B(G^\sigma) = \B(G^\sigma)_{\nn'}$ and each irreducible component of $\B(G^\sigma)_{\nn'}$ is the intersection of each irreducible component of $\B_\nn$ and $\B(G^\sigma)$, i.e.
$$\Irr(\B(G^\sigma)_{\nn'}) = \{ X \cap \B(G^\sigma) \mid X \in \Irr (\B_\nn)\}.$$

We show that indeed (\ref{eq:int}) is satisfied for any $G$ on the list above. First we assume $G=SL_{2n}(\k)$ and $\k^{2n}$ are equipped with a symplectic bilinear form.
By Section \ref{sec:typeA} each $X_i \in \Irr(\B_\nn)$ is given by
$$X_i \colonequals \{[F_j] \in \B \mid \im \nn \subset F_{i} \subset \ker \nn\}.$$
Then $X_i\cap \Irr(\B_\nn) \neq \emptyset$ is equivalent to that there exists $F_i \subset \k^{2n}$ such that $\im \nn \subset F_{i} \subset \ker \nn$, and isotropic if $1\leq i\leq n$ and coisotropic if $n \leq i \leq 2n-1$. But it is true for a generic choice of $\nn$.

If $G=SO_{2n+2}(\k)$, then first suppose that we are given each $\k^{2n+1}$, $\k^{2n+2}$ with a symmetric bilinear form, respectively, 
and an isometry $\k^{2n+1}\hookrightarrow \k^{2n+2}$. Then $G^\sigma \hookrightarrow G$ can be regarded as the monomorphism $SO_{2n+1}(\k) \hookrightarrow SO_{2n+2}(\k)$ induced by the embedding above.
By Section \ref{sec:typeD} each $X_i \subset \Irr \B_\nn$ contains an open dense subset $Y_i$ which is defined by
\begin{align*}Y_i \colonequals \{[F_j] \in \B \mid& \im \nn \not\subset F_i, \textup{there exists a line } \vec{l} \subset \im \nn \textup{ such that }\vec{l}  \subset F_i \subset \ker \nn\}.
\end{align*}
Then $Y_i\cap \Irr(\B_\nn) \neq \emptyset$ is equivalent to that there exists $F_i \subset \k^{2n}$ which satisfies the property above and $F_i \subset \br{e_1, \cdots, e_{n}, f_1, \cdots, f_n, e_{n+1}+f_{n+1}}$. But again it is true for generic choice of $\nn$.

If $G$ is of type $D_4$ or $E_6$, then we argue as follows. Since we chose $\nn$ generically, the intersection of $\B(G^\sigma)$ and each $X \in \Irr(\B_\nn)$ is nonempty if the intersection product $[X]\cdot[\B(G^\sigma)]$ is nonzero in the Chow ring of $\B$. As a Chow ring of a flag variety is canonically isomorphic to its cohomology, it suffices to check if $\iota^*([X])\neq 0$ for $[X] \in H^*(\B)$ where $\iota^*: H^*(\B) \rightarrow H^*(\B(G^\sigma))$ is the pull-back under $\iota: \B(G^\sigma) \hookrightarrow \B$. This can be done by explicit calculation.\footnote{Here each $X \in \Irr(\B_\nn)$ is a Schubert variety, thus its class in $H^*(\B)$ is known. Also one can check that $\iota^*: H^*(\B) \rightarrow H^*(\B(G^\sigma))$ is (at least in these cases) isomorphic to 
$$H^*(\B) \twoheadrightarrow H^*(\B)/\br{\L - \sigma(\L), \L \in \Pic(\B) }.$$}

In sum, we have $\Irr(\B(G^\sigma)_{\nn'}) = \{ X \cap \B(G^\sigma) \mid X \in \Irr (\B_\nn)\}.$ Now we claim that for $X, Y \in \Irr(\B_\nn)$,
\begin{equation}\label{edgeeq}
(X ,Y) \in E(\Gamma_{\nn}) \Leftrightarrow (X \cap \B(G^\sigma) ,Y \cap \B(G^\sigma)) \in E(\Gamma_{\nn'}).
\end{equation}
Indeed, suppose $\dim (X \cap \B(G^\sigma)) \cap (Y \cap \B(G^\sigma)) = \dim \B(G^\sigma)_{\nn'} -1$. Then in particular $X\cap Y$ is nonempty. As we chose $\nn$ generically, $\dim X\cap Y\cap \B(G^\sigma) = \dim X\cap Y -\codim_\B \B(G^\sigma). $ Now if $\dim X\cap Y < \dim \B_\nn -1$, then 
\begin{align*}
\dim X\cap Y\cap \B(G^\sigma) &= \dim X\cap Y -\codim_\B \B(G^\sigma)
\\&< \dim \B_\nn -1-\dim \B + \dim \B(G^\sigma) = \dim \B(G^\sigma)_{\nn'}-1,
\end{align*}
which is absurd. Thus $\dim X\cap Y = \dim \B_\nn-1$ and $(X,Y) \in E(\Gamma_\nn)$. Now the other direction follows from the fact that $|E(\Gamma_\nn)|=|E(\Gamma_{\nn'})|$.

\begin{rmk} It is still not sufficient to give a uniform proof (or verification) which describes the structure of $\Gamma_\nn$ without using the classification of Lie algebras or case-by-case argument, as suggested in \cite{humphreys:mof}. As for the argument in this section, one of the difficulties which hinders avoiding case-by-case argument is to check that the intersection of each irreducible component of Springer fibers and the flag variety of $G^\sigma$ is nonempty.
\end{rmk}

\appendix
\section{The minimal special nilpotent orbit of type $F_4$}

In this section we let $G$ be a simple algebraic group of type $F_4$. We assume that $\nn \in \g$ is a nontrivial element in $X_\mu \subset \n$, where $\mu$ is the highest root among short roots of $G$. We consider $\B_\nn \cap \S(w)$ for each $w\in W$ when it is nonempty, i.e. when $w^{-1}(\mu)>0$ by Lemma \ref{lem:explicit}. Since we only need information of irreducible components and their codimension 1 intersections, we only list $w\in W$ where $\codim_{\B_\nn}\dim (\S(w) \cap \B_\nn)=$ 0 or 1, i.e. $\dim (\S(w) \cap \B_\nn)=12$ or 13 in the following tables. 

Here the first column represents the label of each element $w \in W$ which satisfies the condition (we keep these labels from now on,) the second column gives a reduced expression for each $w\in W$ and the third column is the subset $A_w \subset \Phi^+$ of positive roots which satisfy the condition in Lemma \ref{lem:explicit}(2). Here $(a,b,c,d)$ means $a\alpha_1+b\alpha_2+c\alpha_3+d\alpha_4 \in \Phi$.

\begin{center}
\begin{tabular}{|c|c|c|}
\hline
\multicolumn{3}{|c|}{$\dim (X_w \cap \B_\nn)=\dim \B_\nn=13$}
\\\hline
$w \in W$ & red. exp.& $A_w$
\\\hline
\multirow{3}{*}{$y_\iB$}&\multirow{3}{*}{$s(3231234323123121)$}&
(0, 1, 0, 0), (0, 1, 2, 0), (0, 1, 2, 1), (1, 0, 0, 0), (1, 1, 0, 0),
\\&& (1, 1, 2, 0), (1, 1, 2, 1), (1, 2, 2, 0), (1, 2, 2, 1), (1, 2, 3, 1),
\\&&(1, 2, 4, 2), (1, 3, 4, 2), (2, 3, 4, 2)
\\\hline
\multirow{3}{*}{$y_\iD$}&\multirow{3}{*}{$s(12342312343232)$}&
(0, 0, 1, 0), (0, 1, 0, 0), (0, 1, 1, 0), (0, 1, 2, 0), (1, 0, 0, 0),
\\&&(1, 1, 0, 0), (1, 1, 1, 1), (1, 1, 2, 0), (1, 1, 2, 1), (1, 2, 2, 0),
\\&&(1, 2, 2, 1), (1, 2, 3, 1), (2, 3, 4, 2)
\\\hline
\multirow{3}{*}{$y_\iE$}&\multirow{3}{*}{$s(1234231234323121)$}&
(0, 1, 0, 0), (0, 1, 2, 0), (1, 0, 0, 0), (1, 1, 0, 0), (1, 1, 1, 1),
\\&&(1, 1, 2, 0), (1, 1, 2, 1), (1, 2, 2, 0), (1, 2, 2, 1), (1, 2, 3, 1),
\\&&(1, 2, 4, 2), (1, 3, 4, 2), (2, 3, 4, 2)
\\\hline
\multirow{3}{*}{$y_\iC$}&\multirow{3}{*}{$s(123423123432321)$}&
(0, 0, 1, 0), (0, 1, 0, 0), (0, 1, 2, 0), (1, 0, 0, 0), (1, 1, 0, 0),
\\&&(1, 1, 1, 1), (1, 1, 2, 0), (1, 1, 2, 1), (1, 2, 2, 0), (1, 2, 2, 1),
\\&&(1, 2, 3, 1), (1, 3, 4, 2), (2, 3, 4, 2)
\\\hline
\multirow{3}{*}{$y_\iA$}&\multirow{3}{*}{$s(123423123423121)$}&
(0, 1, 0, 0), (0, 1, 2, 0), (1, 0, 0, 0), (1, 1, 0, 0), (1, 1, 1, 1),
\\&& (1, 1, 2, 0), (1, 1, 2, 1), (1, 2, 2, 0), (1, 2, 2, 1), (1, 2, 2, 2),
\\&& (1, 2, 3, 1), (1, 3, 4, 2), (2, 3, 4, 2)
\\\hline
\multirow{3}{*}{$y_\iF$}&\multirow{3}{*}{$s(231234323123121)$}&
(0, 1, 0, 0), (0, 1, 1, 1), (0, 1, 2, 0), (1, 0, 0, 0), (1, 1, 0, 0),
\\&&(1, 1, 1, 1), (1, 1, 2, 0), (1, 2, 2, 0), (1, 2, 2, 1), (1, 2, 2, 2),
\\&&(1, 2, 3, 1), (1, 3, 4, 2), (2, 3, 4, 2)
\\\hline
\end{tabular}
\end{center}

\begin{center}
\begin{tabular}{|c|c|c|}
\hline
\multicolumn{3}{|c|}{$\dim (X_w \cap \B_\nn)=\dim \B_\nn-1=12$}
\\\hline
$w \in W$ & red. exp.& $A_w$
\\\hline
\multirow{2}{*}{$z_1$}&\multirow{2}{*}{$s(1234231234232)$}&
(0, 1, 0, 0), (0, 1, 1, 0), (0, 1, 2, 0), (1, 0, 0, 0), (1, 1, 0, 0), (1, 1, 1, 1),
\\&& (1, 1, 2, 0), (1, 1, 2, 1), (1, 2, 2, 0), (1, 2, 2, 1), (1, 2, 3, 1), (2, 3, 4, 2)
\\\hline
\multirow{2}{*}{$z_2$}&\multirow{2}{*}{$s(12342312342321)$}&
(0, 1, 0, 0), (0, 1, 2, 0), (1, 0, 0, 0), (1, 1, 0, 0), (1, 1, 1, 1), (1, 1, 2, 0),
\\&&(1, 1, 2, 1), (1, 2, 2, 0), (1, 2, 2, 1), (1, 2, 3, 1), (1, 3, 4, 2), (2, 3, 4, 2)
\\\hline
\multirow{2}{*}{$z_3$}&\multirow{2}{*}{$s(23123432312312)$}&
(0, 1, 0, 0), (0, 1, 1, 1), (1, 0, 0, 0), (1, 1, 0, 0), (1, 1, 1, 1), (1, 1, 2, 0),
\\&&(1, 2, 2, 0), (1, 2, 2, 1), (1, 2, 2, 2), (1, 2, 3, 1), (1, 3, 4, 2), (2, 3, 4, 2)
\\\hline
\multirow{2}{*}{$z_4$}&\multirow{2}{*}{$s(12342312342312)$}&
(0, 1, 0, 0), (1, 0, 0, 0), (1, 1, 0, 0), (1, 1, 1, 1), (1, 1, 2, 0), (1, 1, 2, 1),
\\&&(1, 2, 2, 0), (1, 2, 2, 1), (1, 2, 2, 2), (1, 2, 3, 1), (1, 3, 4, 2), (2, 3, 4, 2)
\\\hline
\multirow{2}{*}{$z_5$}&\multirow{2}{*}{$s(23123423123121)$}&
(0, 1, 0, 0), (0, 1, 2, 0), (1, 0, 0, 0), (1, 1, 0, 0), (1, 1, 1, 1), (1, 1, 2, 0),
\\&&(1, 2, 2, 0), (1, 2, 2, 1), (1, 2, 2, 2), (1, 2, 3, 1), (1, 3, 4, 2), (2, 3, 4, 2)
\\\hline
\multirow{2}{*}{$z_6$}&\multirow{2}{*}{$s(323123423123121)$}&
(0, 1, 0, 0), (0, 1, 2, 0), (1, 0, 0, 0), (1, 1, 0, 0), (1, 1, 2, 0), (1, 1, 2, 1),
\\&&(1, 2, 2, 0), (1, 2, 2, 1), (1, 2, 3, 1), (1, 2, 4, 2), (1, 3, 4, 2), (2, 3, 4, 2)
\\\hline
\multirow{2}{*}{$z_7$}&\multirow{2}{*}{$s(1234323123121)$}&
(0, 0, 0, 1), (0, 1, 0, 0), (1, 0, 0, 0), (1, 1, 0, 0), (1, 1, 1, 1), (1, 1, 2, 0),
\\&&(1, 1, 2, 1), (1, 1, 2, 2), (1, 2, 2, 0), (1, 2, 2, 1), (1, 2, 2, 2), (2, 3, 4, 2)
\\\hline
\multirow{2}{*}{$z_8$}&\multirow{2}{*}{$s(23423123432321)$}&
(0, 0, 1, 0), (0, 1, 0, 0), (0, 1, 1, 1), (0, 1, 2, 0), (0, 1, 2, 1), (1, 1, 0, 0), 
\\&&(1, 1, 2, 0), (1, 2, 2, 0), (1, 2, 2, 1), (1, 2, 3, 1), (1, 3, 4, 2), (2, 3, 4, 2)
\\\hline
\multirow{2}{*}{$z_9$}&\multirow{2}{*}{$s(323123432312321)$}&
(0, 1, 0, 0), (0, 1, 2, 0), (0, 1, 2, 1), (1, 1, 0, 0), (1, 1, 2, 0), (1, 1, 2, 1),
\\&&(1, 2, 2, 0), (1, 2, 2, 1), (1, 2, 3, 1), (1, 2, 4, 2), (1, 3, 4, 2), (2, 3, 4, 2)
\\\hline
\multirow{2}{*}{$z_{10}$}&\multirow{2}{*}{$s(1234231234323)$}&
(0, 0, 1, 0), (0, 1, 1, 0), (0, 1, 2, 0), (1, 0, 0, 0), (1, 1, 0, 0), (1, 1, 1, 1),
\\&&(1, 1, 2, 0), (1, 1, 2, 1), (1, 2, 2, 0), (1, 2, 2, 1), (1, 2, 3, 1), (2, 3, 4, 2)
\\\hline
\multirow{2}{*}{$z_{11}$}&\multirow{2}{*}{$s(1234231234231)$}&
(0, 1, 0, 0), (0, 1, 1, 0), (1, 0, 0, 0), (1, 1, 0, 0), (1, 1, 1, 1), (1, 1, 2, 0),
\\&&(1, 1, 2, 1), (1, 2, 2, 0), (1, 2, 2, 1), (1, 2, 2, 2), (1, 2, 3, 1), (2, 3, 4, 2)
\\\hline
\multirow{2}{*}{$z_{12}$}&\multirow{2}{*}{$s(2342312343232)$}&
(0, 0, 1, 0), (0, 1, 0, 0), (0, 1, 1, 1), (0, 1, 2, 0), (0, 1, 2, 1), (1, 1, 0, 0),
\\&&(1, 1, 1, 0), (1, 1, 2, 0), (1, 2, 2, 0), (1, 2, 2, 1), (1, 2, 3, 1), (1, 3, 4, 2)
\\\hline
\multirow{2}{*}{$z_{13}$}&\multirow{2}{*}{$s(23123432312321)$}&
(0, 1, 0, 0), (0, 1, 1, 1), (0, 1, 2, 0), (1, 1, 0, 0), (1, 1, 1, 1), (1, 1, 2, 0),
\\&&(1, 2, 2, 0), (1, 2, 2, 1), (1, 2, 2, 2), (1, 2, 3, 1), (1, 3, 4, 2), (2, 3, 4, 2)
\\\hline
\multirow{2}{*}{$z_{14}$}&\multirow{2}{*}{$s(23423123423121)$}&
(0, 1, 0, 0), (0, 1, 1, 1), (0, 1, 2, 0), (0, 1, 2, 1), (1, 1, 0, 0), (1, 1, 2, 0),
\\&&(1, 2, 2, 0), (1, 2, 2, 1), (1, 2, 2, 2), (1, 2, 3, 1), (1, 3, 4, 2), (2, 3, 4, 2)
\\\hline
\multirow{2}{*}{$z_{15}$}&\multirow{2}{*}{$s(234231234323121)$}&
(0, 1, 0, 0), (0, 1, 1, 1), (0, 1, 2, 0), (0, 1, 2, 1), (1, 1, 0, 0), (1, 1, 2, 0),
\\&&(1, 2, 2, 0), (1, 2, 2, 1), (1, 2, 3, 1), (1, 2, 4, 2), (1, 3, 4, 2), (2, 3, 4, 2)
\\\hline
\multirow{2}{*}{$z_{16}$}&\multirow{2}{*}{$s(31234323123121)$}&
(0, 0, 1, 1), (0, 1, 2, 0), (1, 0, 0, 0), (1, 1, 0, 0), (1, 1, 1, 1), (1, 1, 2, 0),
\\&&(1, 1, 2, 1), (1, 1, 2, 2), (1, 2, 2, 0), (1, 2, 3, 1), (1, 2, 4, 2), (2, 3, 4, 2)
\\\hline
\multirow{2}{*}{$z_{17}$}&\multirow{2}{*}{$s(323123432312312)$}&
(0, 1, 2, 0), (0, 1, 2, 1), (1, 0, 0, 0), (1, 1, 0, 0), (1, 1, 2, 0), (1, 1, 2, 1),
\\&&(1, 2, 2, 0), (1, 2, 2, 1), (1, 2, 3, 1), (1, 2, 4, 2), (1, 3, 4, 2), (2, 3, 4, 2)
\\\hline
\multirow{2}{*}{$z_{18}$}&\multirow{2}{*}{$s(32312342312321)$}&
(0, 0, 1, 0), (0, 1, 0, 0), (0, 1, 2, 0), (1, 0, 0, 0), (1, 1, 0, 0), (1, 1, 2, 0),
\\&&(1, 1, 2, 1), (1, 2, 2, 0), (1, 2, 2, 1), (1, 2, 3, 1), (1, 3, 4, 2), (2, 3, 4, 2)
\\\hline
\multirow{2}{*}{$z_{19}$}&\multirow{2}{*}{$s(12342312343121)$}&
(0, 1, 2, 0), (1, 0, 0, 0), (1, 1, 0, 0), (1, 1, 1, 1), (1, 1, 2, 0), (1, 1, 2, 1),
\\&&(1, 1, 2, 2), (1, 2, 2, 0), (1, 2, 2, 1), (1, 2, 3, 1), (1, 2, 4, 2), (2, 3, 4, 2)
\\\hline
\multirow{2}{*}{$z_{20}$}&\multirow{2}{*}{$s(12342312343231)$}&
(0, 1, 1, 0), (0, 1, 2, 0), (1, 0, 0, 0), (1, 1, 0, 0), (1, 1, 1, 1), (1, 1, 2, 0),
\\&&(1, 1, 2, 1), (1, 2, 2, 0), (1, 2, 2, 1), (1, 2, 3, 1), (1, 2, 4, 2), (2, 3, 4, 2)
\\\hline
\multirow{2}{*}{$z_{21}$}&\multirow{2}{*}{$s(123423123432312)$}&
(0, 1, 2, 0), (1, 0, 0, 0), (1, 1, 0, 0), (1, 1, 1, 1), (1, 1, 2, 0), (1, 1, 2, 1),
\\&&(1, 2, 2, 0), (1, 2, 2, 1), (1, 2, 3, 1), (1, 2, 4, 2), (1, 3, 4, 2), (2, 3, 4, 2)
\\\hline
\multirow{2}{*}{$z_{22}$}&\multirow{2}{*}{$s(3231234231232)$}&
(0, 0, 1, 0), (0, 1, 0, 0), (0, 1, 1, 0), (0, 1, 2, 0), (1, 0, 0, 0), (1, 1, 0, 0),
\\&&(1, 1, 2, 0), (1, 1, 2, 1), (1, 2, 2, 0), (1, 2, 2, 1), (1, 2, 3, 1), (2, 3, 4, 2)
\\\hline
\multirow{2}{*}{$z_{23}$}&\multirow{2}{*}{$s(1234231234121)$}&
(0, 1, 0, 0), (1, 0, 0, 0), (1, 1, 0, 0), (1, 1, 1, 1), (1, 1, 2, 0), (1, 1, 2, 1),
\\&&(1, 1, 2, 2), (1, 2, 2, 0), (1, 2, 2, 1), (1, 2, 2, 2), (1, 2, 3, 1), (2, 3, 4, 2)
\\\hline
\end{tabular}
\end{center}
We let $Y_i \colonequals \S(y_i) \cap \B_\nn$ and $Z_i \colonequals \S(z_i) \cap \B_\nn$. Then $\Irr\B_\nn=\{\overline{Y_1}, \overline{Y_2}, \cdots, \overline{Y_6}\}$. Now we use the following lemma repeatedly.
\begin{lem}  \label{bruhat_closure}Suppose $s_i \in S$ and  $ w\in W$ satisfy $l(ws_i) =l(w)-1$. Then for any $A \subset \Phi$ which contains $ws_i(\alpha_i)$, we have 
$$\overline{\prod_{\alpha \in A} U_\alpha wB/B} \supset \prod_{\alpha \in A-\{ws_i(\alpha_i)\}} U_\alpha ws_iB/B.$$
Also, $\overline{\prod_{\alpha \in A} U_\alpha wB/B}$ is a union of lines of type $s_i$.
\end{lem}
\begin{proof} We have
\begin{align*}\prod_{\alpha \in A} U_\alpha wB/B &= \left(\prod_{\alpha \in A-\{ws_i(\alpha_i)\}} U_\alpha\right)U_{ws_i(\alpha_i)} w B/B
\\&=\left(\prod_{\alpha \in A-\{ws_i(\alpha_i)\}} U_\alpha\right) ws_i U_{\alpha_i} s_i B/B
\\&=\left(\prod_{\alpha \in A-\{ws_i(\alpha_i)\}} U_\alpha\right) ws_i B s_i B/B,
\end{align*}
thus the claim follows from that $\overline{B s_i B/B} \owns B/B$. The second claim also follows since $\overline{B s_i B/B}$ is a line of type $s_i$.
\end{proof}
By Lemma \ref{bruhat_closure} and direct calculation we obtain the following information.
\begin{enumerate}
\item $\overline{Y_\iB}$ is a union of lines of type $s_1, s_2$, and $s_3$, respectively, but not of type $s_4$.  Also it contains $Z_6, Z_9,$ and $Z_{17}$.
\item $\overline{Y_\iD}$ is a union of lines of type $s_2, s_3$, and $s_4$, respectively, but not of type $s_1$.  Also it contains $Z_1, Z_{10}$, and $Z_{22}$.
\item $\overline{Y_\iE}$ is a union of lines of type $s_1, s_2$, and $s_4$, respectively, but not of type $s_3$.  Also it contains $Z_6, Z_{21}$, and an affine space of dimension 12 in $Y_\iC$.
\item $\overline{Y_\iC}$ is a union of lines of type $s_1, s_3$, and $s_4$, respectively, but not of type $s_2$. Also it contains $Z_2, Z_{18}$, and an affine space of dimension 12 in $Y_\iD$.
\item $\overline{Y_\iA}$ is a union of lines of type $s_1, s_2$, and $s_4$, respectively, but not of type $s_3$. Also it contains $Z_2, Z_4,$ and $Z_5$.
\item $\overline{Y_\iF}$ is a union of lines of type $s_1, s_2$, and $s_3$, respectively, but not of type $s_4$. Also it contains $Z_3, Z_{5}$, and $Z_{15}$.
\end{enumerate}
Thus $\overline{Y_\iC} \cap \overline{Y_\iA}  \supset Z_2, \overline{Y_\iA} \cap \overline{Y_\iF} \supset Z_5, \overline{Y_\iB} \cap \overline{Y_\iE} \supset Z_6$, which means $( \overline{Y_\iC},\overline{Y_\iA}), (\overline{Y_\iA},\overline{Y_\iF}), (\overline{Y_\iB},\overline{Y_\iE}) \in E(\Gamma_\nn)$. Also since each $\overline{Y_\iD}\cap \overline{Y_\iC}$ and $\overline{Y_\iE} \cap \overline{Y_\iC}$ contains an affine space of dimension 12, respectively, $(\overline{Y_\iD},\overline{Y_\iC}), ( \overline{Y_\iE},\overline{Y_\iC}) \in E(\Gamma_\nn)$. We claim that they are all the edges of $\Gamma_\nn$, i.e. $\Gamma_\nn$ is described as follows.

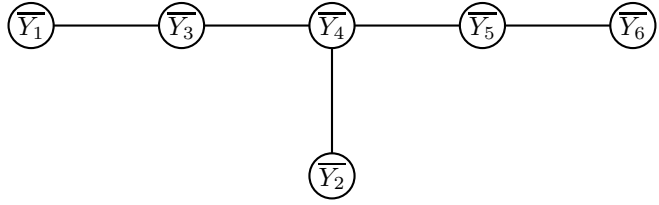
\begin{center}
\begin{tikzpicture}[scale=1]
	\node at (0 cm,0) {$\overline{Y_\iC}$};
	\node at (2 cm,0) {$\overline{Y_\iA}$};
	\node at (4 cm,0) {$\overline{Y_\iF}$};
	\node at (-4 cm,0) {$\overline{Y_\iB}$};
	\node at (-2 cm,0) {$\overline{Y_\iE}$};
	\node at (0 ,-2cm) {$\overline{Y_\iD}$};

    \draw[thick] (0 cm,0) circle (3 mm);
    \draw[thick] (2 cm,0) circle (3 mm);
    \draw[thick] (4 cm,0) circle (3 mm);
    \draw[thick] (-4 cm,0) circle (3 mm);
    \draw[thick] (-2 cm,0) circle (3 mm);
    \draw[thick] (0,-2cm) circle (3 mm);    
    
    \draw[thick] (0.3, 0) -- (1.7,0);
    \draw[thick] (2.3, 0) -- (3.7,0);
    \draw[thick] (-0.3, 0) -- (-1.7,0);
    \draw[thick] (-2.3, 0) -- (-3.7,0);
    \draw[thick] (0, -0.3) -- (0,-1.7);
\end{tikzpicture}
\end{center}

To that end, first we describe the action of $C(\nn) \simeq \Z/2$. By direct calculation any representative of $s_3 \in W$ in $G$ corresponds to $-1 \in \Z/2$, which permutes $Y_i, Z_i$ as follows. (Here $A \leftrightarrow B$ means $s_3(A) = B$ and $s_3(B) =A$.)
\begin{equation}\label{symm:exp}
\begin{gathered}
Y_\iE \leftrightarrow Y_\iA , \quad Y_\iB \leftrightarrow Y_\iF, \quad Z_3 \leftrightarrow Z_{17},  \quad Z_4\leftrightarrow Z_{21},  \quad Z_5\leftrightarrow Z_6,
\\Z_7\leftrightarrow Z_{16}, \quad Z_9 \leftrightarrow Z_{13}, \quad Z_{11} \leftrightarrow Z_{20}, \quad Z_{14}\leftrightarrow Z_{15}, \quad Z_{19} \leftrightarrow Z_{23}
\end{gathered}
\end{equation}
Thus it follows that 
$$s_3(\overline{Y_\iB}) = \overline{Y_\iF},\quad 
s_3(\overline{Y_\iD}) = \overline{Y_\iD},\quad 
s_3(\overline{Y_\iE}) = \overline{Y_\iA}, \quad
s_3(\overline{Y_\iC}) = \overline{Y_\iC},\quad
s_3(\overline{Y_\iA}) = \overline{Y_\iE},\quad 
s_3(\overline{Y_\iF}) = \overline{Y_\iB}.$$ ($s_3$ fixes $\overline{Y_\iC}$ because it is the only component which is a union of lines of type $s_1, s_3,$ and $s_4$, respectively. Similarly $s_3$ fixes $\overline{Y_\iD}$.) It means 
\begin{equation}\label{symm}
\text{$\Gamma_\nn$ is symmetric under the action which switches node $\overline{Y_\iB}$ and $\overline{Y_\iF}$, and node $\overline{Y_\iE}$ and $\overline{Y_\iA}$.}
\end{equation}

Next for each $y_i \in W$ we list elements $y_j, z_j \in W$ which is less than or equal to $y_i$ with respect to the Bruhat order on $W$.
\begin{enumerate}
\item $y_\iB \geq y_\iB, y_\iF, z_3, z_5, z_6, z_7, z_9, z_{13}, z_{16}, z_{17}, z_{18}, z_{22}$
\item $y_\iD \geq y_\iD, z_1, z_{10}, z_{12}, z_{22}$
\item $y_\iE \geq y_\iD, y_\iE, y_\iC, y_\iA,  z_1, z_2, z_4, z_5, z_6, z_8, z_{10}, z_{11}, z_{12}, z_{14}, z_{15}, z_{18}, z_{19}, z_{20}, z_{21}, z_{22}, z_{23}$
\item $y_\iC \geq  y_\iD, y_\iC, z_1, z_2, z_8, z_{10}, z_{11}, z_{12}, z_{18}, z_{20}, z_{22}$
\item $y_\iA \geq y_\iA, z_1, z_2, z_4, z_5, z_{11}, z_{14}, z_{19}, z_{23}$
\item $y_\iF \geq y_\iF, z_3, z_5, z_7, z_{13}, z_{16}$
\end{enumerate}
Note that there is no $y_j$ or $z_j$ which is both less than or equal to $y_\iC$ and $y_\iF$. Thus $(\overline{Y_\iC},\overline{Y_\iF}) \notin E(\Gamma_\nn)$. Similarly, $(\overline{Y_\iD},\overline{Y_\iF}) \notin E(\Gamma_\nn)$. Thus by (\ref{symm}), $(\overline{Y_\iB},\overline{Y_\iC}), (\overline{Y_\iB},\overline{Y_\iD}) \notin E(\Gamma_\nn)$.

We claim that $(\overline{Y_\iD},\overline{Y_\iA}) \not\in E(\Gamma_\nn)$. Suppose the contrary, then as $z_1$ is the only element among $y_j, z_j$ which is both less than or equal to $y_\iD$ and $y_\iA$, we should have $\overline{Y_\iD} \cap \overline{Y_\iA}\supset Z_1$. Since $l(z_1s_1)=l(z_1)+1$, it implies that $\overline{Z_1}$ is not a union of type $s_1$. Now we recall the following lemma.
\begin{lem} \label{spa:line}Let $X$ be a closed irreducible subvariety of $\B_\nn$ for some nilpotent $\nn \in \g$. Suppose there exists $s \in S$ such that for every element $x \in X$ there exists a line of type $s$ which contains $x$ and is contained in $\B_\nn$. If these lines are not all contained in $X$, then the union $Y$ of such lines is a closed irreducible subvariety of $\B_\nn$ and $\dim Y = \dim X+1$.
\end{lem}
\begin{proof} \cite[Lemme II.1.11]{spaltenstein:book}.
\end{proof}

Now let $Y$ be the union of lines of type $s_1$ which pass $\overline{Z_1}$. As $\overline{Y_\iA}$ is a union of lines of type $s_1$ and contains $Z_1$, by the previous lemma it follows that $Y=\overline{Y_\iA}$. By construction we have
$$Y = \overline{\prod_{\alpha \in A_{z_1}} U_\alpha z_1 U_{\alpha_1} s_1 B/B}$$
and $Y \subset \overline{Bz_1s_1B/B}$ as $l(z_1s_1) = l(z_1)+1$. But $z_1s_1< y_\iA$, thus $Y \not\supset Y_\iA$, which is a contradiction. Thus it follows that $( \overline{Y_\iD},\overline{Y_\iA}) \not\in E(\Gamma_\nn)$. By (\ref{symm}), we also have $(\overline{Y_\iD}, \overline{Y_\iE}) \not\in E(\Gamma_\nn)$.

We also claim that $(\overline{Y_\iB},\overline{Y_\iA}) \not\in E(\Gamma_\nn)$. Suppose the otherwise, then as $z_5$ is the only element among $y_j, z_j$ which is both less than or equal to $y_\iA$ and $y_\iB$, we should have $ \overline{Y_\iB}\cap \overline{Y_\iA} \supset Z_5$. But by (\ref{symm:exp}) it means $\overline{Y_\iE}\cap \overline{Y_\iF} \supset Z_6$, which is impossible as $y_\iF \not> z_6$. Now by (\ref{symm}) we also have $(\overline{Y_\iE}, \overline{Y_\iF}) \not\in E(\Gamma_\nn)$.

Thus it only remains to show that $(\overline{Y_\iE},\overline{Y_\iA}), (\overline{Y_\iB}, \overline{Y_\iF}) \not \in E(\Gamma_\nn)$. First assume that $(\overline{Y_\iE},\overline{Y_\iA}) \in E(\Gamma_\nn)$. Then the list of $y_j, z_j$ which is both less than or equal to $y_\iA$ and $y_\iE$ is as follows.
$$y_\iE,y_\iA \geq y_\iA, z_1, z_2, z_4, z_5, z_{11}, z_{14}, z_{19}, z_{23}$$
If $\overline{Y_\iE}$ and $\overline{Y_\iA}$ have a codimension 1 intersection in $Y_\iA$, then by (\ref{symm:exp}) they also have a codimension 1 intersection in $Y_\iE$, which is impossible since $y_\iA < y_\iE$. Also, $\overline{Y_\iA}$ cannot contain $Z_1$ since $\overline{Y_\iD} \supset Z_1$ and we already know that $(\overline{Y_\iD},\overline{Y_\iA}) \not \in E(\Gamma_\nn)$. Also, $\overline{Y_\iE}\cap \overline{Y_\iA}$ cannot contain any of $Z_4, Z_5, Z_{11}, Z_{14}$, respectively, since otherwise by (\ref{symm:exp}) $\overline{Y_\iE}\cap\overline{Y_\iA}$ also contains $Z_{21}, Z_6, Z_{20}, Z_{15}$, respectively, which is impossible as $z_{21}, z_6, z_{20}, z_{15} \not<y_\iA$.

Suppose $\overline{Y_\iE} \cap \overline{Y_\iA} \supset Z_2$. As $l(z_2s_1)=l(z_2)+1$, $\overline{Z_2}$ is not a union of lines of type $s_1$. Thus if we let $Y$ be the union of lines of type $s_1$ which pass $\overline{Z_2}$, then by Lemma \ref{spa:line} we have $Y=\overline{Y_\iE}$. But it is impossible since $Y = \overline{\prod_{\alpha \in \Phi_{z_2}} U_\alpha z_2 U_{\alpha_1}s_1 B/B} \subset Bz_2s_1B/B$ and $z_2s_1 < y_\iE$. Thus if $( \overline{Y_\iE},\overline{Y_\iA}) \in E(\Gamma_\nn)$ then $\overline{Y_\iE},\overline{Y_\iA}$ contains either $Z_{19}$ or $Z_{23}$. By (\ref{symm:exp}), it follows that $\overline{Y_\iE},\overline{Y_\iA}$ contains both $Z_{19}$ and $Z_{23}$.

Now we recall the following fact.\footnote{The author thanks George Lusztig for informing him this result.}
\begin{lem} \label{smoothint}Suppose $\nn \in \g$ be a nilpotent element and $h \in \g$ be a semisimple element such that $[h,\nn] = 2\nn$. Let $\g= \bigoplus_{i \in \Z} \g_i$ be the weight decomposition of $\g$ with respect to $\ad(h)$, i.e. $\ad(h)$ acts on $\g_i$ by multiplication of $i$. Let $P \subset G$ be the parabolic subgroup of $G$ such that $\Lie P = \oplus_{i \geq 0} \g_i$. Then the intersection of $\B_\nn$ with any $P$-orbit of $\B$ is smooth.
\end{lem}
\begin{proof} It follows from \cite[Proposition 3.2]{dclp}. (In the paper $\k=\C$ is assumed, but its proof is still valid in our assumptions.)
\end{proof}
In our case the parabolic subgroup $P \subset G$ which is generated by $\br{B,s_1,s_2, s_3}$ satisfies the condition in the lemma above. Now if we intersect $\B_\nn$ with $Ps(4323412)B/B$, we have
$$\B_\nn \cap (Ps(4323412)B/B) = Y_\iE\sqcup Y_\iA\sqcup Z_4 \sqcup Z_{14}\sqcup Z_{15} \sqcup Z_{19} \sqcup Z_{21} \sqcup Z_{23}\sqcup \text{(dimension $\leq 11$)}.$$
Thus $\overline{Y_\iE}  \cap (Ps(4323412)B/B)$ and $\overline{Y_\iA} \cap (Ps(4323412)B/B)$ are irreducible components of $\B_\nn \cap (Ps(4323412)B/B)$. (One can indeed show that these two are the only irreducible components of $\B_\nn \cap (Ps(4323412)B/B)$, but this is not needed.)

Since $Z_{19}, Z_{23} \subset  \overline{Y_\iE}\cap \overline{Y_\iA} $, it follows that $\overline{Y_\iE}  \cap (Ps(4323412)B/B)$ and $\overline{Y_\iA}  \cap (Ps(4323412)B/B)$ have nonempty intersection. But it is impossible since $\B_\nn \cap (Ps(4323412)B/B)$ is smooth, which implies that every irreducible component is pairwise disjoint. Thus it follows that $(\overline{Y_\iE},\overline{Y_\iA}) \not\in E(\Gamma_\nn)$.

Now we claim that $(\overline{Y_\iB}, \overline{Y_\iF}) \not\in E(\Gamma_\nn)$. Its proof is completely analogous to that of $( \overline{Y_\iE},\overline{Y_\iA}) \not\in E(\Gamma_\nn)$. First the following is the list of $y_j, z_j$ which is both less than or equal to $y_\iB, y_\iF$.
$$y_\iB, y_\iF \geq y_\iF, z_3, z_5, z_7, z_{13}, z_{16}$$
If $\overline{Y_\iB}$ and $\overline{Y_\iF}$ has a codimension 1 intersection in $Y_\iF$, then by (\ref{symm:exp}) they also intersect in $Y_\iB$ with codimension 1, which is impossible since $y_\iF \not>y_\iB$. Also, if $\overline{Y_\iB}\cap\overline{Y_\iF}$ contains any of $Z_3, Z_5, Z_{13}$, respectively, then by (\ref{symm:exp}) it also contains $Z_{17}, Z_{6}, Z_{9}$, which is impossible since $y_\iF \not> z_{17}, z_6, z_9$. Thus it follows that $\overline{Y_\iB}\cap\overline{Y_\iF}$ contains either $Z_7$ or $Z_{16}$. By (\ref{symm:exp}), it means $\overline{Y_\iB}\cap\overline{Y_\iF}\supset Z_7 \sqcup Z_{16}.$

Now if we intersect $\B_\nn$ with $Ps(4323123)B/B$, we have
$$\B_\nn \cap (Ps(4323123)B/B) = Y_\iB \sqcup Y_\iF \sqcup Z_3 \sqcup Z_7 \sqcup Z_9 \sqcup Z_{13} \sqcup Z_{16} \sqcup Z_{17}\sqcup \text{(dimension $\leq 11$)}.$$
Note that $\B_\nn \cap (Ps(4323123)B/B)$ is smooth by Lemma \ref{smoothint}. Also, $\overline{Y_\iB} \cap (Ps(4323123)B/B)$ and $\overline{Y_\iF} \cap (Ps(4323123)B/B)$ are irreducible components of $\B_\nn \cap (Ps(4323123)B/B)$. However, if $Z_7 \sqcup Z_{16} \subset \overline{Y_\iB} \cap  \overline{Y_\iF}$ then $\overline{Y_\iB} \cap (Ps(4323123)B/B)$ and $\overline{Y_\iF} \cap (Ps(4323123)B/B)$ are not disjoint, which is absurd. Thus it follows that $(\overline{Y_\iB}, \overline{Y_\iF}) \not\in E(\Gamma_\nn)$.

\bibliographystyle{amsalphacopy}
\bibliography{minimal}

\end{document}